\documentclass[AMA,Times1COL]{WileyNJDv5} 

\articletype{Article Type}%

\received{Date Month Year}
\revised{Date Month Year}
\accepted{Date Month Year}
\journal{Journal}
\volume{00}
\copyyear{2024}
\startpage{1}

\usepackage{graphicx}%
\usepackage{multirow}%
\usepackage{amsmath,amssymb,amsfonts}%
\usepackage{amsthm}%
\usepackage{mathrsfs}%
\usepackage{xcolor}%
\usepackage{textcomp}%
\usepackage{manyfoot}%
\usepackage{booktabs}%
\usepackage{algorithm}%
\usepackage{algorithmicx}%
\usepackage{algpseudocode}%
\usepackage{listings}%

\DeclareMathAlphabet{\mathpzc}{OT1}{pzc}{m}{it}
\DeclareMathAlphabet{\mathbf}{OT1}{cmr}{bx}{it}

\newcommand{\spK}{{\cal K}}

\newcommand{\R}{\mathbb{R}}

\newcommand{\Rn}{\mathbb{R}^n}
\newcommand{\Rnn}{\mathbb{R}^{n \times n}}

\newcommand{\C}{\mathbb{C}}

\newcommand{\V}{\mathcal{V}}

\newcommand{\prodM}{{\mathbf m}}

\DeclareMathOperator{\nnz}{nnz}
\DeclareMathOperator{\argmin}{arg\,min}

\numberwithin{theorem}{section}
\numberwithin{corollary}{section}
\numberwithin{proposition}{section}
\numberwithin{figure}{section}
\numberwithin{table}{section}
\newtheorem{example}[theorem]{Example}%

\begin{document}

\title{Sketched and truncated polynomial Krylov methods: Evaluation of matrix functions}

\author[1]{Davide Palitta}

\author[2]{Marcel Schweitzer}

\author[1]{Valeria Simoncini}

\authormark{PALITTA \textsc{et al.}}
\titlemark{SKETCHED AND TRUNCATED KRYLOV METHODS FOR MATRIX FUNCTIONS}

\address[1]{\orgdiv{Dipartimento di Matematica}, \orgname{Institution Name}, \orgname{(AM)$^2$, Alma Mater Studiorum - Università di Bologna}, \orgaddress{\postcode{40126} \city{Bologna}, \country{Italy}}}

\address[2]{\orgdiv{School of Mathematics and Natural Sciences}, \orgname{Bergische Universit\"at Wuppertal}, \orgaddress{\postcode{42097} \city{Wuppertal}, \country{Germany}}}

\corres{Davide Palitta \email{davide.palitta@unibo.it}}


\abstract[Abstract]{Among randomized numerical linear algebra strategies, so-called sketching procedures are emerging as effective reduction means to accelerate the computation of Krylov subspace methods for, e.g., the solution of linear systems, eigenvalue computations, and the approximation of matrix functions. While there is plenty of experimental evidence showing that sketched Krylov solvers may dramatically improve performance over standard Krylov methods, many features of these schemes are still unexplored. We derive a new sketched Arnoldi-type relation that allows us to obtain several different new theoretical results. These lead to an improvement of our understanding of sketched Krylov methods{\color{black}, in particular by explaining why the frequently occurring sketched Ritz values far outside the spectral region of $A$ do not negatively influence the convergence of sketched Krylov methods for $f(A)b$. Our findings also help to identify,} among several possible equivalent formulations, the most suitable sketched approximations according to their numerical stability properties. These results are also employed to analyze the error of sketched Krylov methods in the approximation of the action of matrix functions, significantly contributing to the theory available in the current literature.}


\maketitle

\renewcommand\thefootnote{\fnsymbol{footnote}}
\setcounter{footnote}{1}

\section{Introduction}
The field of numerical linear algebra is currently experiencing a shift of paradigm with the advent of fast and scalable randomized algorithms for many core linear algebra problems, the most prominent one probably being the now widely adopted randomized singular value decomposition~\cite{HalkoMartinssonTropp2011}. A specific, important tool that has emerged in recent years in the context of randomized numerical linear algebra is the so-called \emph{sketch-and-solve} paradigm. While initially only used---very successfully---for computations with rather crude accuracy demands, e.g., (low-rank) matrix approximation~\cite{sarlos2006improved,WoolfeLibertyRokhlinTygert2008}, it was recently discovered that---when combined with Krylov subspace methods---it might also be applicable for high accuracy computations as, e.g., the solution of linear systems, eigenvalue computations, and the approximation of matrix functions~\cite{balabanov2022randomized,CortinovisKressnerNakatsukasa2022,GuettelSchweitzer2022,NakatsukasaTropp2021}. The main idea in these methods is to use randomized sketching for reducing the cost of orthogonalization by, e.g., combining it with a truncated Gram--Schmidt process. This often allows a huge boost in performance, as orthogonalization is typically the main bottleneck in polynomial Krylov methods when applied to nonsymmetric problems, particularly on modern high-performance architectures, where each inner product introduces a global synchronization point and communication tends to be the dominating factor. For symmetric matrices, computation of an orthogonal Krylov basis is possible by the short recurrence Lanczos method~\cite{Lanczos1950}, so that the gains one can obtain by reducing the cost of orthogonalization are rather limited. We therefore focus on the nonsymmetric case in the following.

While it was experimentally confirmed that sketched Krylov solvers might dramatically improve performance over standard Krylov methods~\cite{balabanov2022randomized,CortinovisKressnerNakatsukasa2022,GuettelSchweitzer2022,NakatsukasaTropp2021}, many features of these methods are not fully understood yet. In particular, the numerical stability of these methods and how exactly they relate to their non-randomized counterparts needs further analysis. Several recent studies lay a theoretical groundwork with elements of an analysis~\cite{BalabanovNouy2019,CortinovisKressnerNakatsukasa2022,GuettelSchweitzer2022,TimsitGrigoriBalabanov2023}, however, none of these can satisfactorily explain the behavior observed in practice yet, in particular for nonsymmetric problems. {\color{black}E.g., sometimes these methods work incredibly well, essentially reproducing the behavior of the full (non-sketched) Arnoldi method at a fraction of the cost, but in other cases they may fail completely; see, e.g., the experiments reported in~\cite{CortinovisKressnerNakatsukasa2022,GuettelSchweitzer2022}. The currently available theory can neither satisfactorily explain the successful cases, nor the failures.}

In this work, we derive new theoretical results which help better understand the overall behavior of {\color{black}Krylov methods that combine truncated orthogonalization with randomized sketching} and identify which of the several possible (mathematically equivalent) formulations of sketched Krylov approximations can be expected to enjoy the most favorable numerical stability properties. We then use these results to analyze the error of sketched Krylov methods for approximating the action of matrix functions, $f(A)b$, where $A \in \Rnn$ and $b \in \Rn$, significantly contributing to the currently available theory~\cite{CortinovisKressnerNakatsukasa2022,GuettelSchweitzer2022}.

One of our main results is a new Arnoldi-type relation for the
sketched basis {\color{black}(Proposition~\ref{prop:sk_arn_rel})}. This is a general result and it can be used to
analyze many of the sketching approaches obtained so far, from
linear systems to eigenvalue problems. 
We show that full, truncated, and sketched Arnoldi methods are all related by basis transformations and rank-one additions;  {\color{black} see the summary in Table~\ref{tab:summary}}. These explicit relations allow us to derive new algebraic expressions for the distance between the full Arnoldi and sketched approximations to matrix function evaluation, which we analyze in detail {\color{black}(Corollaries~\ref{cor:relation_sketched_full} and~\ref{cor:error_Hd})}. {\color{black}This analysis is based on new results concerning the connection between functions of certain rank-one modified matrices and divided differences (Theorems~\ref{th:errorf} and~\ref{th:ferror_DD}).} We also provide an explanation of the good performance of the so-called sketched-whitened Arnoldi method, in spite of apparently adverse spectral properties of the involved projected matrix{\color{black}; see in particular the family of bounds~\eqref{eq:bound_family}}.


Here is an outline of the paper. In section~\ref{Background material} we recall some background material {\color{black}that will serve as a foundation for our analysis. In particular, we cover} the (full and truncated) Arnoldi method for $f(A)b$ (section~\ref{The Arnoldi relation}), subspace embeddings (section~\ref{Subspace embeddings}) {\color{black}and prior work on sketched Krylov methods (section~\ref{subsec:sketched_fom})}. In section~\ref{subsec:sketched_arnoldi} we show that a new Arnoldi-like relation, that we name the \emph{sketched Arnoldi relation}, holds also when truncated Krylov subspace methods and sketching techniques are combined. Then, in section~\ref{Some general results on functions of rank-1 modified matrices} {\color{black}we show how to rewrite the sketched Arnoldi relation in terms of functions of certain rank-one modified matrices and} present some general results on {\color{black} such matrix functions} that will be employed to relate sketched and non-sketched approximations in section~\ref{Error analysis with respect to the FOM approximation}. The role of eigenvector-related coefficients is further investigated in section~\ref{sec:error_vec}, leading to the counterintuitive observation that the outlying spurious Ritz values generated by the sketched approach do not contribute to the error vector.
In section~\ref{Robust computation of sketched approximations} we discuss some issues related to the stability of the analyzed procedures, while our conclusions are given in section~\ref{Conclusions}.

In the following, the notation $\|\cdot\|$ will be used to denote the Euclidean norm for vectors and the corresponding
induced norm for matrices. Similarly, $\|z\|_\infty:=\max_i|z_i|$ will denote the vector infinity norm whereas $\|T\|_\infty$ the corresponding matrix norm. The identity matrix
will be denoted by $I$, with dimension clear from the context,
while $e_j$ will denote the $j$th column of the identity matrix. We denote the $d$th Krylov subspace corresponding to $A \in \Rnn$ and $b \in \Rn$ by $\spK_d(A,b) = \text{span}\{b, Ab, A^2b, \dots, A^{d-1}b\}$. {Finally,  roman font ($H_d$, $U_d$) will be used for quantities related to a non-orthogonal Arnoldi relation and calligraphic font ($\mathscr{H}_d, \mathscr{U}_d$) for quantities related to a fully orthogonal Arnoldi relation.}
Matlab \cite{matlab7} has been used for all our numerical experiments.

Throughout, exact arithmetic computations will be assumed in our derivations.

\section{Background material}\label{Background material}
In this section we review a few known tools that will be  crucial ingredients in our analysis.

\subsection{The full and truncated Arnoldi method}\label{The Arnoldi relation}
The backbone of most Krylov subspace methods is the Arnoldi process~\cite{Arnoldi1951}, which, given a matrix $A \in \Rnn$, a vector $b \in \Rn$, and a target dimension $d$, computes a nested orthonormal basis $\mathscr{U}_d \in \R^{n \times d}$ of $\spK_d(A,b)$ via the modified Gram--Schmidt orthonormalization and gives rise to the \emph{Arnoldi relation}
\begin{equation}\label{eq:arnoldi_relation}
A\mathscr{U}_d = \mathscr{U}_{d+1}\underline{\mathscr{H}}_d = \mathscr{U}_d \mathscr{H}_d + \mathpzc{h}_{d+1,d}\mathpzc{u}_{d+1}e_d^T,
\end{equation}
where $\underline{\mathscr{H}}_d = [\mathscr{H}_d; \mathpzc{h}_{d+1,d}e_d^T] \in \R^{(d+1) \times d}$, is an unreduced upper Hessenberg matrix containing the orthogonalization coefficients, and
$\mathscr{U}_{d+1}=[\mathscr{U}_d, \mathpzc{u}_{d+1}]$. 

{\color{black} Given the quantities from the Arnoldi relation~\eqref{eq:arnoldi_relation}, one can extract an approximation for $f(A)b$ from the Krylov subspace via
\begin{equation}\label{eq:fullfom}
    f_d^{\textsc{fom}} = \mathscr{U}_d f(\mathscr{H}_d)e_1\|b\|.
\end{equation}
We call this the full Arnoldi (or FOM, in analogy with the Full Orthogonalization Method for linear systems) approximation.}

The main contributions to the computational cost of performing $d$ steps of the Arnoldi method are $d$ matrix vector products with $A$ at a cost of $\mathcal{O}(d \nnz(A))$ floating point operations (flops), assuming that $A$ is sparse with $\nnz(A)$ nonzeros, and the modified Gram--Schmidt orthogonalization which has a cost of $\mathcal{O}(d^2 n)$ flops across all $d$ iterations. Due to its quadratic growth, the cost of orthogonalization quickly dominates the overall cost of the method. In particular, on modern high performance architectures, each inner product introduces a global synchronization point, which leads to a large amount of communication, unless specific ``low-sync'' implementations are used, which come with their own challenges regarding numerical stability; see, e.g, the recent survey~\cite{CarsonLundRozloznikThomas2022}.

{\color{black}An alternative---which is well-explored in the solution of linear systems---is to use a truncated Arnoldi method (see Algorithm~6.6 in~\cite{Saad2003}), i.e., instead of orthogonalizing a new basis vector against all previous basis vectors, one only orthogonalizes it against a fixed number $k$ of previous vectors (with $k = 2$ mimicking the Lanczos method for Hermitian $A$~\cite{Lanczos1950}), thereby overcoming the quadratic growth of the orthogonalization cost. This leads to a \emph{truncated Arnoldi relation}
\begin{equation}\label{eqn:arnoldi_tr}
A U_d = {U}_{d+1} \underline{H}_d,
\end{equation}
with ${U}_{d+1}=[U_d, u_{d+1}]$, where now $\underline{H}_d$ is banded and {\color{black}the columns of $U_d$ are not all mutually orthonormal}. In the following we denote $\underline{H}_d= [H_d; h_{d+1,d}e_d^T]$, where only the last component in the last row is nonzero. Given~\eqref{eqn:arnoldi_tr}, one can compute the truncated FOM approximation
\begin{equation}\label{eq:trfom}
    f_d^{\textsc{tr}} = U_d f(H_d)e_1\|b\|,
\end{equation}
which was, e.g., recently proposed in Algorithm~3.3 of~\cite{CortinovisKressnerNakatsukasa2022}. The price one pays for the reduction in orthogonalization cost when using~\eqref{eq:trfom} is that, typically, convergence is delayed in comparison to the full Arnoldi approximation~\eqref{eq:fullfom}; see, e.g.~\cite{CortinovisKressnerNakatsukasa2022} as well as Example~\ref{ex:whitening} below.}

\subsection{Subspace embeddings}\label{Subspace embeddings}   
{\color{black}The foundations of the sketched Krylov methods discussed in this paper are \emph{subspace embeddings}~\cite{sarlos2006improved,DrineasMahoneyMuthukrishnan2006,woodruff2014sketching}, which embed a low-dimensional subspace $\V$ of $\Rn$ into a smaller space $\mathbb{R}^s$, $s \ll n$, such that inner products (or norms) are distorted in a controlled manner. Specifically, a matrix $S \in \R^{s \times n}$ is called an \emph{$\varepsilon$-subspace embedding for $\V$} (for a given distortion parameter $\varepsilon\in [0,1)$), if
\begin{equation}\label{eq:sketch}
(1-\varepsilon) \| v \|^2 \leq \| S v\|^2 \leq (1+\varepsilon) \|v\|^2,
\end{equation}
for all $v \in \V$. This is equivalent to requiring
\begin{equation}{\label{eq:sketch_innerproduct}}
\langle u, v \rangle - \varepsilon \| u\| \|v\|
            \leq \langle S u, S v \rangle 
            \leq \langle u, v \rangle + \varepsilon \| u\| \|v\|,
\end{equation}
for all $u,v \in \V$ by the parallelogram inequality.}

{\color{black} Constructing $S$ as above requires knowledge of the subspace $\V$ that shall be embedded. In our context, $\V$ will be a Krylov subspace $K_d(A,b)$ which is not known in advance. To still construct an embedding in this case, one can use so-called \emph{oblivious} embeddings. A matrix $S$ is an oblivious $\varepsilon$-subspace embedding for subspaces of dimension $d$, if condition~\eqref{eq:sketch} holds with high probability for \emph{any} subspace $\V$ with $\dim(\V) = d$. To be precise, the inputs for constructing $S$ are the embedding quality $\varepsilon$, the sketching dimension $s$ and an accepted failure probability $\delta$. Based on these, $S$ is constructed using probabilistic methods; see, e.g., \cite{woodruff2014sketching} or section~8 of~\cite{MartinssonTropp2020} for details.}

{\color{black}Frequently, oblivious subspace embeddings are constructed as subsampled trigonometric transforms,
\begin{equation}\label{eq:subsampled_trigonometric}
S = \sqrt{\frac{s}{n}}DNE,
\end{equation}
where $E \in \Rnn$ is a diagonal matrix with Rademacher entries (i.e., the diagonal entries are randomly chosen as $\pm 1$ with equal probability), $D \in \R^{s \times n}$ contains $s$ randomly selected rows of the identity matrix, and $N \in \Rnn$ is a trigonometric transform (e.g., discrete Fourier transform, discrete cosine transform or Walsh--Hadamard transform). One can show that~\eqref{eq:subsampled_trigonometric} with $s = \mathcal{O}(\varepsilon^{-2}(d + \log\frac{n}{\delta})\log\frac{d}{\delta})$ is an oblivious $\varepsilon$-subspace embedding with failure probability $\delta$. With a careful implementation, such $S$ can be applied to a vector in $\mathcal{O}(n\log s)$ flops. In practice, selecting the smaller sketching dimension $s = \mathcal{O}(\varepsilon^{-2}\frac{d}{\delta})$ appears to work very well, except for a few academic examples. We refer to, e.g.,~\cite{tropp2011improved} and section~9 of~\cite{MartinssonTropp2020} for a discussion and experimental evidence.}

{\color{black}As the required sketching dimension $s$ depends quadratically on the embedding quality $\varepsilon$, one typically needs to accept a rather crude embedding quality, such as $\varepsilon = 1/\sqrt{2}$ or $\varepsilon = 1/2$, in order to avoid excessive growth of computational cost.}

{\color{black} To conclude this subsection, we mention that a} subspace embedding $S$ induces a semidefinite inner product
\begin{equation}\label{eq:semidef_inner_product}
    \langle u, v\rangle_S := \langle Su, Sv\rangle.
\end{equation}
From the embedding property, it directly follows that when restricted to $\mathcal{V}$, $\langle \cdot,\cdot\rangle_S$ is an actual inner product; see, e.g.,~section~3.1 in~\cite{BalabanovNouy2019}.

In the remainder of the paper, we often simply call $S$ a ``sketching matrix'', thereby meaning that it is an $\varepsilon$-subspace embedding for a certain Krylov subspace (which will be clear from the context) with a suitable value of $\varepsilon$.

\subsection{The sketched Arnoldi method}\label{subsec:sketched_fom}
{\color{black} As already mentioned in the introduction, a rather novel approach for reducing the orthogonalization cost in Krylov subspace methods via truncated orthogonalization \emph{without} sacrificing convergence speed is to combine the method with oblivious subspace embeddings, leading to so-called \emph{sketched Krylov methods}; see~\cite{NakatsukasaTropp2021} for an application to linear systems and eigenvalue problems as well as~\cite{CortinovisKressnerNakatsukasa2022,GuettelSchweitzer2022} for approximating the action of matrix functions. Similar ideas have also been discussed in~\cite{BalabanovNouy2019} in the context of model order reduction.}

{\color{black}The main idea of the methods in~\cite{GuettelSchweitzer2022,NakatsukasaTropp2021} is to first cheaply construct a non-orthogonal Krylov basis and then subsequently approximately orthogonalize it, working just with sketched vectors of size $s \ll n$ (a process known as \emph{basis whitening}). We explain the methodology for approximations of $f(A)b$, closely following section~2 of~\cite{GuettelSchweitzer2022}. Given a non-orthogonal Krylov basis $U_d$ resulting from a truncated Arnoldi process and a sketching matrix $S$, the ``sketched FOM'' (sFOM) approximation for $f(A)b$ can be defined as
\begin{equation}\label{eq:sfom}
    f_d^{\textsc{sk}} = U_d f\left((SU_d)^\dagger SAU_d\right)(SU_d)^\dagger Sb,
\end{equation}
where $(SU_d)^\dagger$ denotes the Moore-Penrose pseudoinverse of $SU_d$. The basis whitening procedure then consists of computing a thin QR decomposition $S U_d=Q_d T_d$ of the sketched basis and performing the replacements
\begin{equation}\label{eq:basis_whitening_transformations}
SU_d \leftarrow Q_d,  \ \ {}\quad SAU_d  \leftarrow (SAU_d) T_d^{-1}, \ {}\quad U_d \leftarrow U_d T_d^{-1} 
\end{equation}
Note that the matrix $U_d T_d^{-1}$ from the last replacement need not be formed\footnote{Forming this matrix would incur a cost of $\mathcal{O}(d^2n)$ and thus be as expensive as performing the full Arnoldi method in the first place, so this needs to be avoided.}, because inserting~\eqref{eq:basis_whitening_transformations} into~\eqref{eq:sfom} yields
\begin{equation}\label{eq:sfom2}
    f_d^{\textsc{sk}} = U_d \left(T_d^{-1}f\left(Q_d^\ast SAU_dT_d^{-1}\right)Q_d^\ast Sb\right),
\end{equation}
which can be evaluated by first applying $T_d^{-1}$ via back substitution and then multiplying by $U_d$. }

{\color{black}The motivation for using~\eqref{eq:basis_whitening_transformations}--\eqref{eq:sfom2} is that the  transformed basis $U_dT_d^{-1}$ appearing in~\eqref{eqn:WSarnoldi} is orthogonal with respect to the inner product $\langle\cdot,\cdot\rangle_S$ and
is known to be extremely well-conditioned in terms of its {spectral} condition number. This is stated in the following result, which directly follows from Corollary~2.2 in~\cite{balabanov2022randomized}; see also section~4 in~\cite{sohler2011subspace} for related, earlier results in a probabilistic setting.}

{\color{black}%
\begin{proposition}\label{prop:whitenend_basis_conditioning}
If $S$ is an $\varepsilon$-subspace embedding of $\spK_d(A,b)$, then the 2-norm condition number of $U_dT_d^{-1}$ is bounded as
\begin{equation}\label{eq:whitenend_basis_conditioning}
\kappa_2(U_dT_d^{-1}) \leq \sqrt{\frac{1+\varepsilon}{1-\varepsilon}}.
\end{equation}
\end{proposition}}

{\color{black}For example, for a practically reasonable distortion parameter $\varepsilon = 1/\sqrt{2}$, Proposition~\ref{prop:whitenend_basis_conditioning} guarantees that $\kappa_2(U_dT_d^{-1}) \leq 1+\sqrt{2} \approx 2.4142$.}

{\color{black}Numerical experiments in section~5 of~\cite{GuettelSchweitzer2022} show that indeed, the approximation~\eqref{eq:sfom2} often converges remarkably similar to the full Arnoldi approximation despite using the cheaper truncated orthogonalization. It is also observed that the method may sometimes fail completely, though.}

{\color{black}As an alternative to using sketching as a means to retrospectively improve the conditioning of a locally orthogonal Krylov basis, another recently explored avenue for reducing orthogonalization cost and communication in Krylov methods via sketching is to perform a \emph{full} Gram--Schmidt orthogonalization, but with respect to the semidefinite inner product $\langle\cdot,\cdot\rangle_S$ defined in~\eqref{eq:semidef_inner_product} instead of the Euclidean inner product; cf., e.g.,~\cite{balabanov2022randomized,TimsitGrigoriBalabanov2023} for use of this technique in the context of solving linear systems by FOM or GMRES. In Algorithm~3.1 of~\cite{CortinovisKressnerNakatsukasa2022} this approach is extended to approximating the action of matrix functions. By using this approach, inner products in the orthogonalization become cheaper (as they involve only vectors of length $s \ll n$), but the asymptotic cost of the orthogonalization remains at $\mathcal{O}(d^2n)$, just as for full Arnoldi. This approach is therefore mainly attractive when working in a highly parallel computing environment where inner products are the main bottleneck due to expensive communication.}

\section{The sketched Arnoldi relation}\label{subsec:sketched_arnoldi}
{\color{black} In the following, we focus on better understanding the behavior of the sketched Arnoldi approximation~\eqref{eq:sfom} (or equivalently~\eqref{eq:sfom2}). To do so, we start by deriving} a sketched analogue of the Arnoldi relation~\eqref{eq:arnoldi_relation}, {\color{black} which allows to rewrite~\eqref{eq:sfom2} in a way that is both easier to interpret and---more importantly---allows to derive several new theoretical results in later sections of this paper.}

\begin{proposition} \label{prop:sk_arn_rel}
{\color{black} Let ${U}_{d+1}, \underline{H}_d$ be the quantities from a truncated Arnoldi relation~\eqref{eqn:arnoldi_tr} and
let} $SU_{d+1}=Q_{d+1} T_{d+1}$ be a thin QR decomposition with 
\[
Q_{d+1}=[Q_d,q] \text{ and } T_{d+1}=
\begin{bmatrix}
T_d& t\\
0^T &\tau_{d+1}
\end{bmatrix}.
\]
Then the following Arnoldi-like formula holds:
\begin{equation}\label{eq:sketched_arnoldi}
S AU_d = S U_{d} (H_d + r e_d^T) + \tau_{d+1} h_{d+1,d} q e_d^T,
\end{equation}
with $r=h_{d+1,d}T_d^{-1}t$ and $q\perp S U_d$.
\end{proposition}

\begin{proof}
Le $h^T = h_{d+1,d}e_d^T$. We have
\begin{eqnarray}
S AU_d &=& S U_{d+1}\underline{H}_d = Q_{d+1} T_{d+1} \underline{H}_{d} 
 {=} Q_{d+1} \begin{bmatrix} T_d H_d + t h^T \\ [0^T,\tau_{d+1}] \underline{H}_d  \end{bmatrix}\nonumber\\
&=&  Q_d T_d (H_d + T_d^{-1}t h^T)+ \tau_{d+1} q h^T 
=S U_d (H_d + T_d^{-1}t h^T)+ \tau_{d+1} q h^T.\label{eqn:sk-rel}
\end{eqnarray}
The result follows by inserting the definition of $r$.
\end{proof}

The action of sketching appears only in the vector $r$ and in the scalar $\tau_{d+1}$, while the matrix $H_d + r e_d^T$ differs from $H_d$ only in its last column, with a rank-one modification. This will have a quite interesting impact in the following derivations. 

\begin{remark}
Relations of the form~\eqref{eqn:arnoldi_tr} do not only arise in the context of the truncated Arnoldi method, but also in case of other recurrences for generating non-orthogonal Krylov bases, e.g., employing Chebyshev or Newton polynomials; see~\cite{PhilippeReichel2012}. Proposition~\ref{prop:sk_arn_rel} thus also applies to these methods, but we restrict to truncated Arnoldi here for ease of exposition and because it appears to be the most widely used basis generation method in sketched Krylov methods.
\end{remark}

{\color{black}In order to relate Proposition~\ref{prop:sk_arn_rel} to the basis whitening procedure~\eqref{eq:basis_whitening_transformations} outlined in section~\ref{subsec:sketched_fom}, first note that this} procedure was already {\color{black}implicitly} employed in the derivation of the sketched Arnoldi relation {\color{black}due to occurrence of the thin QR decomposition}. The {\color{black}underlying} orthogonalization step can be {\color{black}even} better appreciated by rewriting (\ref{eqn:sk-rel}) as
\begin{eqnarray}\label{eq:whitening_Arnoldi}
S AU_d T_d^{-1} &=& Q_d T_d (H_d + T_d^{-1}t h^T)
T_d^{-1} + \tau_{d+1} q h^T T_d^{-1},
\end{eqnarray}
from which we obtain the {\it whitened-sketched} Arnoldi relation (WS-Arnoldi in short)
\begin{eqnarray}\label{eqn:WSarnoldi}
S A(U_d T_d^{-1})=
Q_d (\widehat H_d + \widehat t e_d^T) +
{\color{black}\frac{h_{d+1,d}\tau_{d+1}}{\tau_d}} q e_d^T,
\end{eqnarray}
where $\widehat H_d=T_d H_d T_d^{-1}$ and $\widehat{t} = (h_{d+1,d}/\tau_{d})t$.

The inclusion of the orthogonalization coefficients $T_d$, and thus the use of~\eqref{eqn:WSarnoldi} in place of~\eqref{eq:sketched_arnoldi},
seems to be very beneficial to the numerical computation. Indeed, computing functions of $H_d + r e_d^T$ appears to be quite sensitive to round-off, due to the growing components of $r$, which lead to poor balancing; see the discussion in Remark~\ref{rem:balancing}. 
Note however, that 
\[
\sqrt{\frac{1-\varepsilon}{1+\varepsilon}}\kappa_2(T_d) \leq \kappa_2(U_d) \leq \sqrt{\frac{1+\varepsilon}{1-\varepsilon}}\kappa_2(T_d),
\]
so that $T_d$ must necessarily be ill-conditioned whenever $U_d$ is ill-conditioned. As $U_d$ results from a truncated orthogonalization process, it will become more and more ill-conditioned as $d$ increases. 
Despite this fact, inversion of $T_d$ can be expected to not be very problematic in the context of sketched Krylov methods, as we discuss in section~\ref{Robust computation of sketched approximations}.

{\color{black}The result of Proposition~\ref{prop:sk_arn_rel} can be used to analyze sketched Krylov methods for the approximation of matrix functions $f(A)b$. Some first steps regarding the analysis of these methods were done in section~3.2 of~\cite{GuettelSchweitzer2022} and section~4 of~\cite{CortinovisKressnerNakatsukasa2022}. Still, sketched Krylov methods for $f(A)b$ are far from being fully understood from a theoretical perspective. We highlight two particular issues that have not been well explored so far. First, sketching might (and often does) introduce ``sketched Ritz values'' which lie very far outside the spectral region of $A$. However, these tend to mostly not negatively influence the performance of the methods. Second, there exist many different formulations for how to compute sketched Krylov approximations, all of them involving computations with potentially very ill-conditioned matrices; cf.~also the discussion in the preceding paragraph. Still, even in the presence of severe ill-conditioning and rank deficiency, sketched Krylov methods can work well in practice. Throughout the next sections, we approach both these issues and derive new theoretical results which allow to better capture the properties of the sketched Arnoldi approximation and better explain the observed behavior. We start by discussing the relation between the sketched Arnoldi approximation and functions of certain rank-one modified matrices in section~\ref{Some general results on functions of rank-1 modified matrices}.}
{\color{black}
\section{The impact of rank-one modifications on matrix functions}\label{Some general results on functions of rank-1 modified matrices}
}
{\color{black}The formula~\eqref{eq:sfom2} defining the sketched Arnoldi approximation is quite hard to interpret and thus not very insightful. In particular, the meaning of the small matrix $Q_d^\ast SAU_dT_d^{-1}$ at which $f$ is evaluated is rather unclear. We now show how to rewrite it using the tools from section~\ref{subsec:sketched_arnoldi} to obtain formulas that are both easier to interpret and more suitable for a subsequent theoretical analysis}.

{\color{black}From the sketched Arnoldi relation~\eqref{eq:sketched_arnoldi} and the fact that $(SU_d)^\dagger = T_d^{-1}Q_d^*$, we directly obtain by straightforward algebraic manipulations that~\eqref{eq:sfom} simplifies to
\begin{equation}\label{eq:sfom_Hrh}
    f_d^{\textsc{sk}} = U_d f(H_d+re_d^T)e_1\|b\|,
\end{equation}
a form much more reminiscent of the standard and truncated FOM approximations~\eqref{eq:fullfom} and~\eqref{eq:trfom}. Using the WS-Arnoldi relation~\eqref{eqn:WSarnoldi}, we can also write
\begin{equation}\label{eq:sfom_whitened_hessenberg}
f_d^{\textsc{sk}} =  U_d T_d^{-1}f\left(\widehat{H}_d + \widehat{t}e_d^T\right)e_1\|Sb\|.
\end{equation}}

{\color{black} As shown above, the sFOM approximation to $f(A)b$ involves a rank-one modification to $H_d$, or a similar matrix.
To better understand the role of such modification}, we begin by stating a quite general result, which provides an explicit relationship between $f(M) \in \R^{d \times d}$ and functions of a rank-one modification of the form $f(M+ve_d^T)$. We assume in the following that $M+ve_d^T$ is diagonalizable with eigenvalue decomposition
\begin{equation}\label{eq:eigenvalue_decomposition}
M + v e_d^T=X\Lambda X^{-1}, \textnormal{ where } \Lambda={\rm diag}(\lambda_1, 
\ldots, \lambda_d),
\end{equation}
and define the corresponding function
\begin{eqnarray}\label{eq:g}
g_v(t)=\sum_{i=1}^d \alpha_i \beta_i f[t,\lambda_i],
\end{eqnarray}
where 
\begin{equation}\label{eq:divdiff_firstorder}
    f[t,\lambda]= \begin{cases}
    f^\prime(t), & \textnormal{if } t = \lambda, \\ 
    \frac{f(t)-f(\lambda)}{t-\lambda}, & \textnormal{otherwise},
    \end{cases}
\end{equation}
denotes a divided difference and $\alpha_i = e_d^T X e_i$, $\beta_i = e_i^T X^{-1} e_1$ with $X$ from the eigenvalue decomposition~\eqref{eq:eigenvalue_decomposition} of $M+ve_d^T$. The subscript in $g_v$ is meant to keep track of the rank-one modification used to generate the spectral quantities defining $g_v$. {\color{black} The proof of this result is given in Appendix~\ref{appendixB}.}

\begin{theorem}\label{th:errorf}
Let $M + v e_d^T$ have the eigendecomposition~\eqref{eq:eigenvalue_decomposition} and assume that $f$ is defined in a region containing the eigenvalues of $M$ and $M + v e_d^T$. Then
\[
f(M + v e_d^T)e_1 - f(M)e_1 = g_v(M) v,
\]
where $g_v$ is defined in~\eqref{eq:g}.
\end{theorem}

The formula derived in Theorem~\ref{th:errorf} provides insightful information on the
role of the (undesired) eigenvalues of the reduced sketched matrix; see section~\ref{sec:error_vec}.
However, this expression for the matrix function $g_v$ makes it difficult to 
analyze and measure the error between the different function evaluations{\color{black}; see also~\cite{IlicTurnerSimpson2010} where such observations are made for $g_0$}.
The following {\color{black}reformulation leads} to a more
penetrating interpretation, and {\color{black}it} may be read off as a different definition of $g_v${\color{black}; see Appendix~\ref{appendixB} for its proof. Notice that} the second definition in the theorem below is stated in terms of higher-order divided differences, which---among other representations---can be defined recursively based on~\eqref{eq:divdiff_firstorder} (see, e.g., section~5 in~\cite{DeBoor2005}) or via the Cauchy integral formula, as
\begin{equation}\label{eq:integral_higher_order_divdiff}
f[z_1,\dots,z_{d+1}] = \frac{1}{2 \pi \imath} \int_\Gamma \frac{f(z)}{(z-z_1)\cdots(z-z_{d+1})} dz;
\end{equation}
see section~9 in~\cite{DeBoor2005}.

\begin{theorem}\label{th:ferror_DD}
Let $M$ be upper Hessenberg and let $\prodM_d:=\prod_{j=1}^{d-1}M_{j+1,j}$, {\color{black}where $M_{j+1,j}$ denotes the $(j+1,j)$th entry of $M$}. Assume that $M + v e_d^T$ has the eigendecomposition~\eqref{eq:eigenvalue_decomposition}, that all eigenvalues $\{\lambda_1, \ldots, \lambda_d\}$ of $M+v e_d^T$ are simple and that $f$ is defined in a region containing the eigenvalues of $M$ and $M + v e_d^T$. Then
\begin{equation}\label{eqn:ferrnew1}
f(M + v e_d^T)e_1 - f(M)e_1 = 
\prodM_d 
\, \sum_{i=1}^d \frac{1}{\omega^\prime(\lambda_i)} f[M,\lambda_i] v,
\end{equation}
{ with $\omega(z)= \displaystyle\prod_{i=1}^d (z - \lambda_i)$,} from which it also holds
\begin{equation}\label{eqn:ferrnew2}
f(M + v e_d^T)e_1 - f(M)e_1 = 
\prodM_d \, f[M, \lambda_1, \ldots, \lambda_d] v,
\end{equation}
where $f[M, \lambda_1, \ldots, \lambda_d]$ is the order-$(d+1)$ divided difference
of the function $f$.
\end{theorem}

\begin{remark}
Higher-order divided differences have been used before in matrix function computations, see, e.g.,~\cite{EiermannErnst2006,TalEzer2007}, where they are employed for representing the error of the (standard) Arnoldi approximation in the context of restarted methods {\color{black}(where, in our notation, the function $g_0$ appears, as no rank-one modifications occur in that setting)}. As the authors of~\cite{EiermannErnst2006,TalEzer2007} observe, the numerical computation of higher-order divided differences can quickly become unstable, so that restarted algorithms based on these error representations did not turn out to be universally applicable, in particular not for ``hard'' problems. Except for~\cite{IlicTurnerSimpson2010}, where a slightly more stable representation based on first-order divided differences was used, similar to our Theorem~\ref{th:errorf}, there has not been too much subsequent work exploring this connection. In section~4 of~\cite{TalEzer2007}, it is also brief\/ly discussed how to use a divided difference representation in order to estimate the norm of the Arnoldi error when $A$ is (close to) normal.

While it might not be generally useful for numerical computations, in the following we rather use the divided difference representation as a theoretical tool which allows to compare different methods and obtain insight into their qualitative behavior. 
\end{remark}

A few comments on the previous results are in order.
By comparing Theorem~\ref{th:errorf} and (\ref{eqn:ferrnew1}), it follows that
for any $i \in \{1, \ldots, d\}$ it holds
$$
\alpha_i \beta_i = \frac{\prod_{j=1}^{d-1}M_{j+1,j}}{\omega^\prime(\lambda_i)}, 
$$
providing some insight into the magnitude of the product $\alpha_i\beta_i$ for outlying
eigenvalues $\lambda_i$. 
In section~\ref{sec:error_vec} the role of $\beta_i$ will be explored in more details.

We now use the above analysis in order to obtain a general bound for the norm of the difference between a function of a matrix and its rank-one modification. This bound emphasizes the role of the function $f$ and the domain on which it acts. The role of the vector $v$ defining the rank-one modification will be discussed in more detail in section~\ref{Error analysis with respect to the FOM approximation}.
The results in Theorem~\ref{th:ferror_DD}
can be conveniently interpreted by using a matrix analog of the well-known relation
between divided differences and differentiation. 
For an analytic function $h$ defined on the field of values of $M$,
we will also make use of the following bound proved by Crouzeix and Palencia \cite{Crouzeix.Palencia.17},
\begin{equation}\label{eq:crouzeix}
\|h(M)\|_2 \le (1+\sqrt{2}) \max_{z\in F(M)} |h(z)|.
\end{equation}
We recall that the field of values of a square matrix $M\in{\mathbb C}^{d\times d}$ is defined as $F(M)=\{z\in{\mathbb C}\,:\, z=(x^*Mx)/(x^*x),\, x\in {\mathbb C}^d,\,x\neq 0\}$.

The following bound can be 
derived from the \emph{Genocchi--Hermite representation} of divided differences; see, e.g., \cite{DeBoor2005}, Appendix~B.16 of~\cite{Higham2008} or Problem~(6.1.43) in \cite{HornJohnson1991}.
\begin{proposition}\label{prop:divdiff_factorial}
Let $D$ be a closed, convex set in the complex plane containing $z_1,\dots,z_{d+1}$ and let $f$ be analytic on $D$. Then
$$
|f[z_1, \ldots, z_{d+1}]| \le \frac{\max_{\zeta \in D} |f^{(d)}(\zeta)|}{d!}.
$$
\end{proposition}

{\color{black}From this proposition, we obtain the following bound which will be used in section~\ref{sec:error_vec} below for the sketched Arnoldi approximation.

\begin{corollary}\label{cor:error_divdiff_general}
Let the assumptions of Theorem~\ref{th:ferror_DD} hold and let $D \subseteq \C$ be convex and contain $F(M)$ as well as the nodes $\lambda_1,\dots,\lambda_d$. Then,
\begin{equation}
\|f(M + v e_d^T)e_1 - f(M)e_1\| \le \prodM_d (1+\sqrt{2})\frac{\max_{\zeta \in D} |f^{(d)}(\zeta)|}{d!}\, \|v\|. \label{eq:bound_gMw}
\end{equation}
\end{corollary}
\begin{proof}
Denote
\begin{equation}\label{eq:h}
h = f[\cdot,\lambda_1,\dots,\lambda_d],
\end{equation}
i.e., $h(z)= f[z,\lambda_1,\dots,\lambda_d]$. From~\eqref{eqn:ferrnew2}, we then have
\begin{equation}\label{eqn:bound_h1}
\|f(M + v e_d^T)e_1 - f(M)e_1\| = \prodM_d \|h(M)\|. 
\end{equation}
Due to the assumptions on $D$, we can apply Proposition~\ref{prop:divdiff_factorial} and further use~\eqref{eq:crouzeix} to bound
\begin{equation}\label{eqn:bound_h}
\|h(M)\| \leq (1+\sqrt{2}) \max_{z \in F(M)} |h(z)| \le (1+\sqrt{2})\frac{\max_{\zeta \in D} |f^{(d)}(\zeta)|}{d!}.
\end{equation}
Putting~\eqref{eqn:bound_h1} and~\eqref{eqn:bound_h} together concludes the proof.
\end{proof}}

For bounded $D$, estimates of the form~\eqref{eq:bound_gMw} are insightful for functions such as the exponential, for which it is easy to see that the term $\prodM_d/d!$ will eventually go to zero with growing $d$, as $\prodM_d\le (\|M\|_{\infty})^d$, while all derivatives $f^{(d)}$ remain bounded. Interestingly, this asymptotic behavior agrees with classical bounds for Krylov subspace approximations for the exponential function~\cite{GallopoulosSaad1992}. In the next
section we will show that the error formulas in Theorem~\ref{th:errorf} and Theorem~\ref{th:ferror_DD} are very general, and can  be used to relate the approximation obtained by the fully orthogonal basis with that computed by either the truncated or sketched bases. Hence, the bound in~\eqref{eq:bound_gMw}, although rather pessimistic in general, ensures that the two quantities $f(M + v e_d^T)e_1$ and $f(M)e_1$ will coincide in exact arithmetic, for large enough $d$.

\begin{example}\label{ex:toepl}
We consider the following problem  with the Toeplitz matrix $M_d$, where the subscript stresses the dependence on the dimension $d$,
{\noindent\scriptsize
\begin{verbatim}
for d=5:5:30
    M=toeplitz( [-4, 2, zeros(1,d-2)], [-4, 1/2, 1/2, zeros(1,d-3)]);
    rng(21*pi); e1=eye(d,1);
    v=flipud(linspace(1,20,d))'.*randn(d,1);
    Mhat=M+v*[zeros(1,d-1),1];
    y1=expm(M)*e1;
    y1s=expm(Mhat)*e1;
end
\end{verbatim}
}
\vskip 0.05in
{These data were chosen for full reproducibility. We also consider the case where the Toeplitz matrix $M$ is substituted by the true upper Hessenberg matrix obtained with the full Arnoldi recurrence at the given iteration {\tt d}. For this latter case, the Arnoldi iteration was used with $b$ equal to a unit norm constant vector of dimension 2500, and $A$ the finite difference discretization of the operator $\Delta u -10xu_x -10yu_y$, subject to zero boundary conditions. }
{The left plot of } 
Figure~\ref{fig:asympt} shows the error norm $\|\exp(M_d + v_d e_d^T)e_1 - \exp(M_d)e_1\|$ and the bound $(1+\sqrt{2})\|v{\color{black}_d}\|\prodM_d/d!$.
Since $\prodM_d$ is moderate, the convergence slope of the asymptotic curve nicely follows that of the true error. {The right plot shows an analogous performance for the matrix obtained via an Arnoldi procedure.}
\end{example}

\begin{figure}[tbhp]
 \centering
\includegraphics[height=.38\textwidth,width=0.48\textwidth]{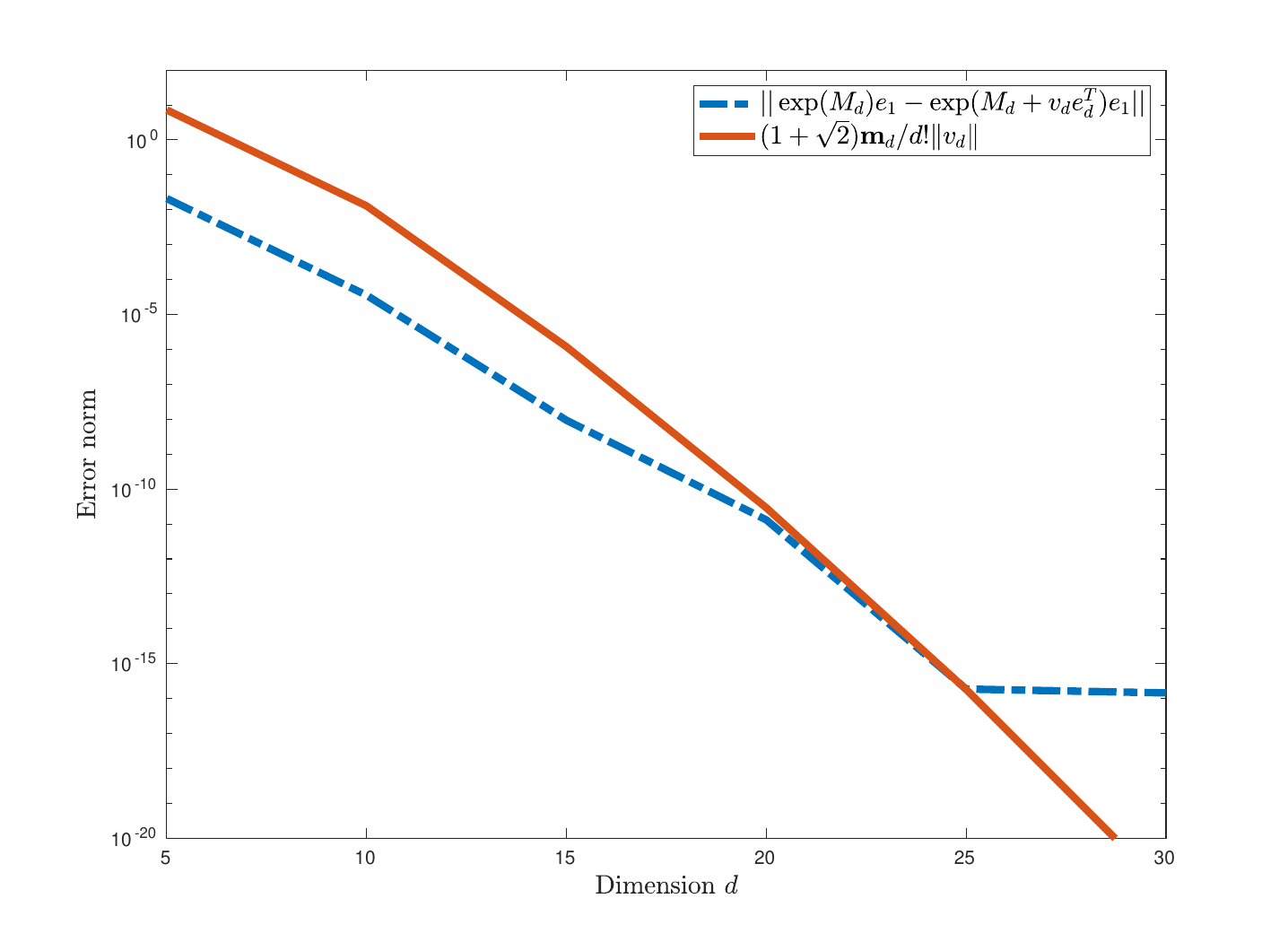}
\includegraphics[height=.38\textwidth,width=0.48\textwidth]{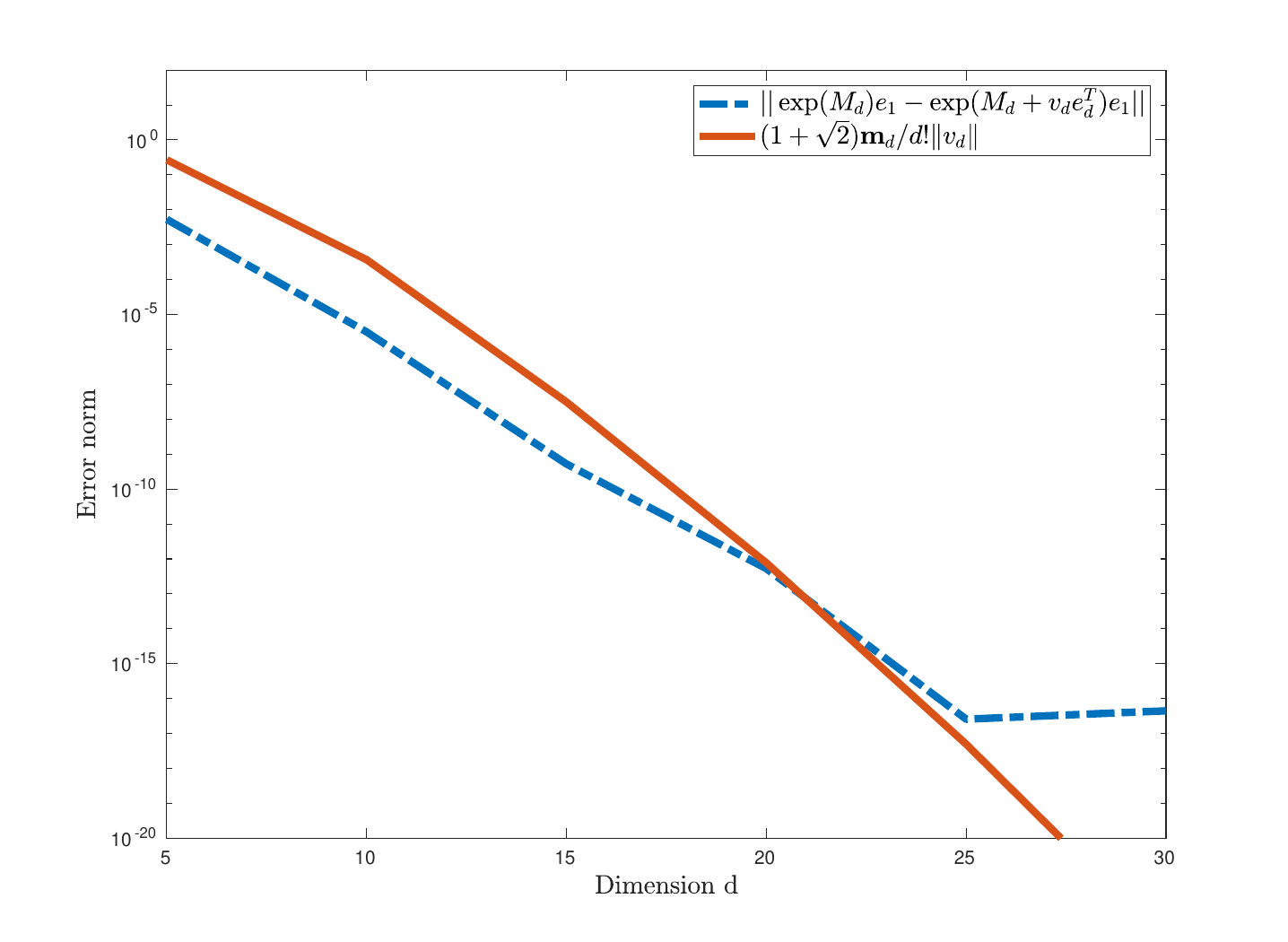}
  \caption{Example~\ref{ex:toepl}. True error $\|\exp(M_d + v_d e_d^T)e_1 - \exp(M_d)e_1\|$
(dashed line)  as the dimension $d$ grows, and bound
  $(1+\sqrt{2})\|v{\color{black}_d}\|\prodM_d/d!$ (solid line). {Left: Toeplitz matrix. Right: $M_d$ obtained by the Arnoldi recurrence with $A$ from discretized PDE problem.}
\label{fig:asympt}}
\end{figure}

For functions other than the exponential the bound in (\ref{eq:bound_gMw}) may not be significant. 
As an example, for $z^{-1/2}$ the formula for the derivative contains a factor that essentially grows as the factorial.
To provide further insight, we resort to representations involving more general
tools that take into account the properties of divided differences in complex domains.
More precisely, we can use a result proved by Curtiss in 1962. We first recall that
for a bounded region  $D$ of the complex plane whose boundary is an admissible Jordan curve, it is possible to define a function
$z=\chi(\omega)=\mu \omega +\mu_0+ \frac{\mu_1}{\omega}+\frac{\mu_2}{\omega^2}+\cdots$, where $\mu>0$ is called the capacity of the Jordan curve, 
and $\omega$ is defined outside the unit circle. The function $\chi$ maps conformally the region outside the
unit circle into the exterior of $\bar D$.
Setting $\omega=r \exp(i \theta)$, there exists a maximum value of $r$, that is $\rho>1$, such that
$f$ is analytic on the region  between the exterior of $D$ and
the curve $C_r=\{z=\chi(r \exp(i \theta) )\}$, with $\theta \in [0, 2\pi]$.
We refer
to \cite{Curtiss.62} for a more comprehensive  description of the notation.
\begin{theorem}(Theorem~4.1 in~\cite{Curtiss.62})\label{thm:curtiss}
With the above notation, let $D$ be a bounded region of the complex plane whose boundary is an admissible Jordan curve,
and let $\mu$ be the capacity of $\partial D$.
Let $f$ be analytic on $C=\bar D$ and let the sequence of $\{\lambda_1, \ldots, \lambda_d\}$ be equidistributed
on $\partial D$. Let $\rho>1$ be the largest value such that $f$ is analytic in the
extended curve $C_{\rho}$ containing $C$.
Then for any $r$, $1< r <\rho$, 
$$
|f[\lambda_1, \ldots, \lambda_d]| \le 
\frac{\ell_{C_r}}{2\pi(\mu r)^d}\max_{z\in C_r}|f(z)|,
$$
where
$\ell_{C_r}$ is the length of the curve
$C_r$.
\end{theorem}

{Combining the result of Theorem~\ref{thm:curtiss} with \eqref{eqn:ferrnew2} and the first inequality in~\eqref{eqn:bound_h}, we obtain the following bound, where $D$ can be taken as the convex set above,}
\begin{eqnarray*}
\|f(M + v e_d^T)e_1 - f(M)e_1\| &\le&
\prodM_d (1+\sqrt{2}) \max_{z \in F(M)} 
|f[z,z_1, \ldots, z_d]| \, \|v\| \\
&\sim& 
\ell_{C_r}
\frac{\prodM_d}{(\mu r)^d}
\max_{z\in C_r}|f(z)| \, \|v\|.
\end{eqnarray*}

Clearly, the above analysis assumes an idealized setting (e.g., in reality, the $\lambda_i$ will not be equidistributed on $\partial D$) and additionally, the obtained bound cannot be expected to be sharp and descriptive of the actual difference of the two approximations. It should merely be seen as a first hint at what kind of results might be possible to obtain for general $f$.
However, such  more refined results will require advanced tools from complex analysis which are well beyond the scope of the current work.

\section{Error analysis with respect to the FOM approximation}\label{Error analysis with respect to the FOM approximation}
We can now use Theorem~\ref{th:errorf} to relate the different possible sketched and truncated Krylov approximations to each other, as they all result from rank-one modifications of one another. Note that for the sketched Krylov approximation, we will use the representation~\eqref{eq:sfom_Hrh} for analysis purposes (as it directly uses the matrix $H_d$), although the representation~\eqref{eq:sfom_whitened_hessenberg} is better suited for actual computations, as we will illustrate in Example~\ref{ex:whitening} below. Before stating our results, let us introduce yet another way of writing the sketched or truncated Krylov approximations: As both $U_d$ and $\mathscr{U}_d$ contain nested bases of $\spK_d(A,b)$, there exists a nonsingular upper triangular matrix $\mathscr{T}_d$ such that $U_d = \mathscr{U}_d\mathscr{T}_d$; the matrix $\mathscr{T}_d$ can, e.g., be obtained via a Gram--Schmidt based QR decomposition. We write
\begin{equation}\label{eq:curlyT}
\mathscr{T}_{d+1}=
\begin{bmatrix}
\mathscr{T}_{d}& \mathpzc{t}\\
0^T &\widehat{\tau}_{d+1}
\end{bmatrix}.
\end{equation}
Using this, we can write the sketched FOM approximation in terms of the orthogonal Krylov basis $\mathscr{U}_d$ as
\begin{equation}\label{eq:sFOM_orth_basis1}
f_d^{\textsc{sk}} = \mathscr{U}_d f(\mathscr{T}_dH_d\mathscr{T}_d^{-1} + \mathscr{T}_d r e_d^T \mathscr{T}_d^{-1})\mathscr{T}_de_1 \|b\|.
\end{equation}
Now, from~\eqref{eqn:arnoldi_tr} and the fact that the columns of $\mathscr{U}_d$ are orthonormal it follows by straightforward algebraic manipulations that 
\begin{equation}\label{eq:sFOM_orth_basis2}
\mathscr{T}_dH_d\mathscr{T}_d^{-1} = \mathscr{H}_d - h_{d+1,d}\mathscr{U}_d^Tu_{d+1}e_d^T\mathscr{T}_d^{-1}.
\end{equation}
Thus, the matrix at which $f$ is evaluated in~\eqref{eq:sFOM_orth_basis1} is a rank-one modification of $\mathscr{H}_d$. Further simplification is possible by noting that the first columns of $U_d$ and $\mathscr{U}_d$ are identical, so the top-left entry of $\mathscr{T}_d$ can be chosen to be $1$ and furthermore, $e_d^T\mathscr{T}_d^{-1} = e_d^T/\widehat{\tau}_d$. Using these observations together with~\eqref{eq:sFOM_orth_basis2} allows to rewrite~\eqref{eq:sFOM_orth_basis1} as
\begin{equation}\label{eq:sFOM_orth_basis3}
f_d^{\textsc{sk}} = \mathscr{U}_d f(\mathscr{H}_d + \widehat{r} e_d^T)e_1 \|b\|,
\end{equation}
where we introduced the shorthand notation 
\begin{align}
\widehat{r} &= (\mathscr{T}_d r - h_{d+1,d}\mathscr{U}_d^Tu_{d+1})/\widehat{\tau}_d \nonumber\\
&= (\mathscr{T}_d T_d^{-1}t - \mathscr{U}_d^Tu_{d+1})(h_{d+1,d}/\widehat{\tau}_d) \label{eq:r_hat2}\\
&= (\mathscr{T}_d T_d^{-1}t - \mathpzc{t})(h_{d+1,d}/\widehat{\tau}_d),\label{eq:r_hat}
\end{align} 
with $\mathpzc{t}$ coming from~\eqref{eq:curlyT}.

By following the same steps for the truncated FOM approximation, we find the representation
\begin{equation}\label{eq:trfom_orth_basis}
f_d^{\textsc{tr}} = \mathscr{U}_d f(\mathscr{H}_d - h_{d+1,d}\mathpzc{t}e_d^T)e_1\|b\|.
\end{equation}

{\color{black} Due to the many different possible formulations, we summarize all important matrix function approximations  (both in the truncated and fully orthonormal bases) in Table~\ref{tab:summary} to provide the reader with an overall picture. Among the various sketched forms we have discussed, we report here only those
that appear to be most reliable for practical purposes.\footnote{For a matrix analysis, the truncated-Arnoldi form of the sketched approximation, namely $f_d^{\textsc{sk}} = U_d f\left({H}_d + {r}e_d^T\right)e_1\|b\|$, would have several advantages. However, as we illustrate in Example~\ref{ex:whitening} below, this form may showcase a quite erratic numerical behavior, compared with the stabilized one reported in the table.}

\begin{table}
\caption{\color{black}Summary of analyzed matrix  function approximations, in the truncated and fully orthonormal bases. \label{tab:summary}  }
\centering
\begin{tabular}{|c|cc|cc|}
\hline
approx        & \multicolumn{2}{c|}{truncated basis $U_d$}  &  \multicolumn{2}{c|}{fully orthonormal basis $\mathscr{U}_d$} \\
\hline
 & & & & \\
$f_d^{\textsc{fom}}$ &  --&  & $\mathscr{U}_d f(\mathscr{H}_d )e_1\|b\|$ & (\ref{eq:fullfom})\\
 & & & & \\
$f_d^{\textsc{tr}}$ & $U_d f(H_d)e_1\|b\|$ & (\ref{eq:trfom}) & 
$\mathscr{U}_d f(\mathscr{H}_d - h_{d+1,d}\mathpzc{t}e_d^T)e_1\|b\|$&  (\ref{eq:trfom_orth_basis}) \\
 & & & & \\
$f_d^{\textsc{sk}}$ &  
$U_d T_d^{-1}f\left(\widehat{H}_d + \widehat{t}e_d^T\right)e_1\|Sb\|$&  (\ref{eq:sfom_whitened_hessenberg}) & 
$\mathscr{U}_d f(\mathscr{H}_d + \widehat{r} e_d^T)e_1 \|b\|$ & (\ref{eq:sFOM_orth_basis3})       \\
 & & & & \\
\hline
\end{tabular}
\end{table}}

{\color{black}For studying the a-priori error with respect to the ideal FOM approximation, writing the truncated and the sketched approximations in terms of the fully orthonormal basis $\mathscr{U}_d$ (third column
in Table~\ref{tab:summary}) leads to more insightful formulas that we also exploit in the results below. It should be kept in mind, however, that in our setting the formulation of sketched and truncated approximations via the fully orthonormal basis is not computationally affordable.}

\vskip 0.1in
\begin{corollary}\label{cor:relation_sketched_full}
The sketched and full Arnoldi approximations $f_d^{\textsc{sk}}$ and $f_d^{\textsc{fom}}$ from~\eqref{eq:sfom_Hrh} and~\eqref{eq:fullfom} fulfill
\begin{equation}\label{eq:fullfom_nonorthogonal_basis}
f_d^{\textsc{sk}}-f_d^{\textsc{fom}} = \mathscr{U}_d g_{\widehat{r}}(\mathscr{H}_d)\widehat{r}\|b\|,
\end{equation}
where $\widehat{r}$ is defined in~\eqref{eq:r_hat} and $g_{\widehat{r}}$ is defined in~\eqref{eq:g}, with respect to the eigendecomposition of $\mathscr{H}_d+\widehat{r}e_d^T$.
\end{corollary}
\begin{proof}
The result follows by applying Theorem~\ref{th:errorf} with $M = \mathscr{H}_d$ and $v = \widehat{r}$.
\end{proof}

\vskip 0.1in
\begin{corollary}\label{cor:relation_truncated_full}
The truncated and full Arnoldi approximations $f_d^{\textsc{tr}}$ and $f_d^{\textsc{fom}}$ from~\eqref{eq:trfom} and~\eqref{eq:fullfom} fulfill
\[
f_d^{\textsc{tr}}-f_d^{\textsc{fom}} = \mathscr{U}_d g_y(\mathscr{H}_d)y\|b\|,
\]
where $y = - h_{d+1,d}\mathpzc{t}$ and $g_y$ is defined in~\eqref{eq:g}, with respect to the eigendecomposition of $\mathscr{H}_d+ye_d^T$.
\end{corollary}
\begin{proof}
The result follows by applying Theorem~\ref{th:errorf} with $M = \mathscr{H}_d$ and $v = - h_{d+1,d}\mathpzc{t}$.
\end{proof}

\vskip 0.1in
There are several possibilities for comparing the sketched and truncated Arnoldi approximations to each other. We present only the one based on the banded Hessenberg matrix $H_d$ and the non-orthogonal basis $U_d$ in the following. It would also be possible, using~\eqref{eq:sFOM_orth_basis3} and~\eqref{eq:trfom_orth_basis}, to state a result in terms of the matrices $\mathscr{U}_d$ and $\mathscr{H}_d$ from the full Arnoldi process, but this is less appropriate here as neither of the two methods uses these matrices in practice.
\begin{corollary}\label{cor:relation_sketched_truncated}
The truncated and sketched Arnoldi approximations $f_d^{\textsc{tr}}$ and $f_d^{\textsc{sk}}$ from~\eqref{eq:trfom} and~\eqref{eq:sfom_Hrh} fulfill
\[
f_d^{\textsc{tr}}-f_d^{\textsc{sk}} = U_d g_r(H_d)r\|b\|,
\]
where $g_r$ is defined in~\eqref{eq:g}, with respect to the eigendecomposition of $H_d+re_d^T$.
\end{corollary}

\begin{proof}
The result follows from Theorem~\ref{th:errorf} applied with $M = H_d$ and $v = r$.
\end{proof}

When bounding the distance between the different approximations, the results of Theorem~\ref{thm:curtiss} and Proposition~\ref{prop:divdiff_factorial} can be used to handle the matrix functions occurring on the right-hand sides of the preceding corollaries by bounding $\|g_v(\mathscr{H}_d)\|$. Such bounds on the norm of the matrix function are only useful if we also have appropriate bounds available for the vectors {\color{black}$v$}. In the following, we particularly discuss the vector {\color{black}$v = \widehat{r}$} from Corollary~\ref{cor:relation_sketched_full} which relates the sketched to the full Arnoldi approximation.

Consider the equation~\eqref{eq:r_hat2}. Using the fact that $t = Q_{d+1}^*Su_{d+1}$ together with $T_d^{-1}Q_d^* = (SU_d)^\dagger$, we find the representation
\begin{align*}
\widehat{r} &= (\mathscr{T}_d(SU_d)^\dagger Su_{d+1}-\mathscr{U}_d^Tu_{d+1})(h_{d+1,d}/\widehat{\tau}_d) \\
            &= \mathscr{U}_d^T(U_d(SU_d)^\dagger Su_{d+1}-u_{d+1})(h_{d+1,d}/\widehat{\tau}_d),
\end{align*}
where we have used $\mathscr{T}_d = \mathscr{U}_d^TU_d$ for the second equality. The vector $(SU_d)^\dagger Su_{d+1}$ appearing above is the solution vector $x_S$ of the sketched least squares problem
\[
x_S = \argmin_{x \in \mathbb{R}^d} \| S(U_dx-u_{d+1})\|.
\]
From established theory covering sketching methods for least squares problems~\cite{sarlos2006improved}, it is known that the solution $x_S$ of the sketched problem cannot have a significantly higher residual than the solution of the original problem if the embedding quality is high enough. Precisely, denoting by $x_\star = U_d^\dagger u_{d+1}$ the solution of the non-sketched least squares problem, we have the following inequality relating the residuals of $x_S$ and $x_\star$,
\[
\|U_dx_S-u_{d+1}\| \leq \sqrt{\frac{1+\varepsilon}{1-\varepsilon}}\|U_dx_\star-u_{d+1}\|;
\]
see~Equation~(2.3) in~\cite{NakatsukasaTropp2021}, which is adapted from results in~\cite{sarlos2006improved,woodruff2014sketching}.\footnote{We remark that the result is stated slightly differently in~\cite{NakatsukasaTropp2021}: Precisely, the scalar factor on the right hand side of the inequality does not involve a square root, because it is based on a sketching property that is stated in terms of $\|v\|$, and not in terms of $\|v\|^2$ as our inequality~\eqref{eq:sketch}.}

Further, due to the fact that $U_d = \mathscr{U}_d\mathscr{T}_d$ is a QR decomposition, we have that $\|U_dx_\star - u_{d+1}\| = |\widehat{\tau}_{d+1}|$.
Since $\|U_dx_\star - u_{d+1}\| =\|(I-\mathscr{U}_d\mathscr{U}_d^T) u_{d+1}\|$, 
the quantity $|\widehat{\tau}_{d+1}|$ measures the contribution of the new
vector $u_{d+1}$ to the expansion of the Krylov subspace.

Putting all this together and using the fact that $\mathscr{U}_d$ has orthonormal columns, we find
\begin{equation}\label{eq:norm_r_hat}
\|\widehat{r}\| \leq h_{d+1,d}\left|\frac{\widehat{\tau}_{d+1}}{\widehat{\tau}_d}\right|\sqrt{\frac{1+\varepsilon}{1-\varepsilon}}.
\end{equation}

As long as  the space keeps growing, we expect that 
$\left|\frac{\widehat{\tau}_{d+1}}{\widehat{\tau}_d}\right| \approx 1$,
so that an overall result of the following form can be obtained from Proposition~\ref{prop:divdiff_factorial}.

{\color{black}%
\begin{corollary}\label{cor:error_Hd}
Let $\mathscr{H}_d \in \C^{d \times d}$ be the upper Hessenberg matrix from the (full) Arnoldi relation~\eqref{eq:arnoldi_relation} and let $D \subseteq \C$ be a convex set that contains $F(\mathscr{H}_d)$ and the eigenvalues of $\mathscr{H}_d + \widehat{r}e_d^\top$. Then
\begin{equation}\label{eq:bound_fsk_D}
\|f_d^{\textsc{sk}}-f_d^{\textsc{fom}} \|\lesssim (1+\sqrt{2})\sqrt{\frac{1+\varepsilon}{1-\varepsilon}}\|b\| \frac{\prod_{j=1}^{d-1} \mathscr{H}_{d+1,d}}{d!} \max_{z \in D} |f^{(d)}(z)| .
\end{equation}    
\end{corollary}}

The bound {\color{black}in Corollary~\ref{cor:error_Hd}} depends on the set $D$, and in general, $D$ might be much larger and have less favorable properties than $F(\mathscr{H}_d)$ (which is guaranteed to lie within $F(A)$). For example, a straightforward choice of $D$ would be the convex hull of $F(\mathscr{H}_d) \cup F(\mathscr{H}_d + \widehat{r}e_d^\top)$. In that case, the bound~\eqref{eq:bound_fsk_D} qualitatively reproduces available results from the literature, as, e.g., Theorem~4.3 in~\cite{CortinovisKressnerNakatsukasa2022} and Corollary~2.4 in~\cite{GuettelSchweitzer2022}. {\color{black}These results rely on the Crouzeix-Palencia theorem applied directly to $f_d^{\textsc{sk}}-f_d^{\textsc{fom}}$ and thus necessarily need to work with a region $D$ in the complex plane that includes the field of values of the sketched and projected matrix. In contrast, we apply the Crouzeix-Palencia theorem only to $h$ from~\eqref{eq:h}, which is evaluated at $\mathscr{H}_d$. The influence of the rank-one modified matrix $\mathscr{H}_d + \widehat{r}e_d^\top$ thus only occurs through its eigenvalues, thus giving us more freedom in the choice of $D$. In particular, it allows us to use the results presented in section~\ref{sec:error_vec} below to argue why eigenvalues far outside $F(\mathscr{H}_d)$ do not negatively influence the behavior of the method. This is in line with  experimental evidence reported in~\cite{CortinovisKressnerNakatsukasa2022, GuettelSchweitzer2022} as well as in our Example~\ref{ex:whitening} below.}


\section{On the error vector and the role of outlying spurious Ritz values}\label{sec:error_vec}
{In the previous section we wrote the error between different methods in terms of the ``error'' vector
(see Theorem~\ref{th:errorf})
$$
g_v(M)v  = \sum_{i=1}^d \alpha_i \beta_i f[M,\lambda_i]v,
$$
where $\alpha_i = e_d^T X e_i$, $\beta_i = e_i^T X^{-1} e_1$ with $X$ from the eigenvalue
decomposition~\eqref{eq:eigenvalue_decomposition}.}

The magnitude of the $\alpha_i, \beta_i$ and of the divided differences needs further analysis. In general, and for large $\|v\|$, we do not expect the eigenvalues of $\mathscr{H}_d$ and $\mathscr{H}_d + v e_d^T$ (or $H_d$ and $H_d + v e_d^T$) to be close to each other. We now discuss this again in the particular situation of Corollary~\ref{cor:relation_sketched_full}, i.e., for $v = \widehat{r}$, noting that the derivation works similarly in the other cases. 

For any nonzero $d\times d$ lower Hessenberg matrix $N$, we define
\begin{equation}\label{eqn:gammai}
\gamma_{\lambda,N} := 
\min_{i=1, \ldots, d-1}\frac{|N_{i,i}-\lambda|}{|N_{i,i+1}|}.
\end{equation}
This quantity allows us to analyze the contribution of the eigenvalues of $\mathscr{H}_d+\widehat r e_d^T$ in the error vector.
We recall here that the eigenvalues of $\mathscr{H}_d+\widehat r e_d^T$ are the same as those of
${H}_d+ r e_d^T$.
As soon as ill-conditioning arises in the generation of the basis $U_d$, the entries of ${r}$ start to grow in magnitude.
In this case, some of the eigenvalues $\lambda_i$ may be far from the field
of values of $\mathscr{H}_d$, as $F(\mathscr{H}_d+ \widehat{r} e_d^T)$ is not included in $F(\mathscr{H}_d)$, whereas
others may be either close or within $F(\mathscr{H}_d)$. 
In the following, we show that whenever $\|r\|$ grows,
the quantity $\beta_i$ is positively affected by becoming correspondingly small.
 
We next assume that $\gamma_{\lambda_i,\mathscr{H}_d}>1$. In particular, this is the case if the eigenvalue $\lambda_i$ is far from  $F(\mathscr{H}_d)$, relative to $\|\mathscr{H}_d\|$.
It appears that a significantly distant eigenvalue
does not interfere with the actual error, so that the corresponding term
in the sum of $g_v$ becomes negligible, in spite of possibly large values of $f[\mathscr{H}_d,\lambda]$. This argument is justified next.  

\begin{proposition}\label{rem:growth_yk}
{Let $N$ be} a $d\times d$ {\it lower} Hessenberg matrix with 
$N_{k,k+1}\ne 0$, for all $k$, and  let $\widehat N=N+ e_d v^*$ for some $v$.
Then for each eigenpair $(\lambda, y)$ of $\widehat N$ with
$\|y\|=1$ it holds, denoting the components $y = (\eta_1,\dots, \eta_d)^T$, that
$$
| \eta_{k+1}| \le  
 \frac{|N_{k,k} - \lambda|}{|N_{k,k+1}|} |\eta_k|  + 
\frac{\|N_{k,1:k-1}\|}{|N_{k,k+1}|}, \quad k = 1,\dots,d-1,
$$
and
\[
|(N_{d,d}+v_d) - \lambda| |\eta_d|  \le \|v_{1:d-1}+N_{d,1:d-1}\|.
\]
\end{proposition}

\begin{proof}
{Let $(\lambda,y)$ be an eigenpair of $\widehat N$, that is} $\widehat N y = \lambda y$. For each row $k$ with $2\le k<d$
it holds 
$$
N_{k,k+1} \eta_{k+1} =-(N_{k,k} - \lambda) \eta_k  + 
N_{k,1:k-1} \eta_{1:k-1}.
$$
We assume that not all indexed components of $y$ are zero, otherwise the bound follows trivially.
Hence,
$$
| \eta_{k+1}| \le  
\frac{|N_{k,k} - \lambda|}{|N_{k,k+1}|} |\eta_k|  + 
\frac{\|N_{k,1:k-1}\|}{|N_{k,k+1}|}\, \|\eta_{1:k-1}\|
\le \frac{|N_{k,k} - \lambda|}{|N_{k,k+1}|} |\eta_k|  + 
\frac{\|N_{k,1:k-1}\|}{|N_{k,k+1}|}.
$$
Equating the last row yields the second bound.
\end{proof}

{A few comments are in order.}
Let 
$\phi_{k}(\lambda)=\prod_{j=1}^{k} 
 \frac{(N_{j,j} - \lambda)}{N_{j,j+1}}
$
be the polynomial  of degree $k$ associated with the  diagonal and superdiagonal of $N$. 

If $\frac{|N_{j,j} - \lambda|}{|N_{j,j+1}|}>1$ for each $j$, then the proposition above implies  that
$$
|\eta_{k+1}| \approx |\phi_k(\lambda)|, \qquad k<d.
$$
This means that the components tend to grow as a power of $k$, that is
\begin{equation}\label{eq:yk_power}
|\eta_k| \approx O(\gamma_{\lambda,N}^{k-1}).
\end{equation}
Due to the  unit norm of $y$, this means that if $\lambda$ is sufficiently far from 
the diagonal elements of $N$,
the first components of $y$ are tiny,
and the subsequent components grow as a growing power of such distance, with the last component of $y$ being $O(1)$ (see Example~\ref{ex:toepl}).
Finally, we observe that the inequality $\frac{|N_{j,j} - \lambda|}{|N_{j,j+1}|}>1$ for each $j$ may also hold if $|\lambda|$ is much larger than $|N_{j,j}|$ and $|N_{j,j+1}|$; in this case, the increasing property of the components is an intrinsic property of the eigenpair, and it may occur also for $\lambda$ not far from the field of values of $N$.

The result {of Proposition}~\ref{rem:growth_yk} is quite general.
To apply it to our setting for the error
formula, we take $N=\mathscr{H}_d^*$. 
Except for normalization, $y=y_i$ is the left
eigenvector of $\mathscr{H}_d$ corresponding to the eigenvalue $\lambda_i$ for some $i$. 
In particular, $y_i^*e_1=e_i^TX^{-1}e_1$. If $\lambda_i$ 
is far away from the diagonal elements of $\mathscr{H}_d^*$,
the
contribution of $\beta_i=e_1^T X^{-1}e_i=(\bar y_i)_1$ becomes extremely limited, that is, 
$$
|\beta_i|\approx 
O( \gamma_{\lambda,N}^{1-d}).
$$ 
This property is crucial because it ensures the following counterintuitive result: although the rank-one modification can dramatically affect the spectrum, the wildly varying eigenvalues are not going to contribute to the error.

The following example illustrates the eigenvector growth property. 
\vskip 0.05in

\begin{example}\label{ex:toepl1}
We consider the {Toeplitz} data in Example~\ref{ex:toepl} for fixed
$d=100$, with $\mathscr{H}_d^*$ playing the role of $N$ and $(\mathscr{H}_d+ \widehat{r} e_d^T)^*$ that of $N+e_d v^T$.
The left plot in Figure~\ref{fig:toeplitz1} shows the spectral quantities
associated with the two matrices, including the boundary of $F(\mathscr{H}_d)$ for later reference. The plot
clearly reports the presence of many eigenvalues of $\mathscr{H}_d+ \widehat{r} e_d^T$ outside the
field of values of $\mathscr{H}_d$, and also falling into the right half-plane.
The right plot shows the components absolute value of the two left eigenvectors  of
$\mathscr{H}_d+ \widehat{r} e_d^T$ corresponding to the eigenvalues pointed to in the left plot.
The two additional lines report the values of $|\phi_k(\lambda)|$ for the
two eigenvalues; the vectors containing these quantities
have been scaled so that the last component has unit value. The eigenvector components behavior is as expected, and
it is accurately predicted by the quantity $|\phi_k(\lambda)|$.
We notice that the first several components of each eigenvector 
have a magnitude at the level of machine epsilon.
In exact arithmetic their values would follow the steep slope of the
other components. 
\end{example}

\begin{figure}[tbhp]
 \centering
\includegraphics[height=.35\textwidth,width=0.48\textwidth]{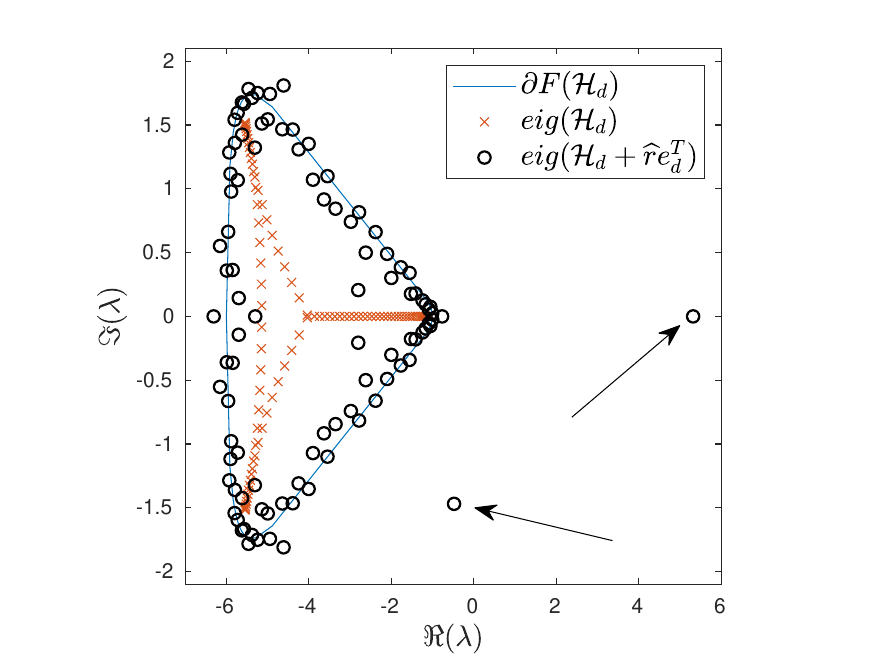}
\includegraphics[height=.35\textwidth,width=0.48\textwidth]{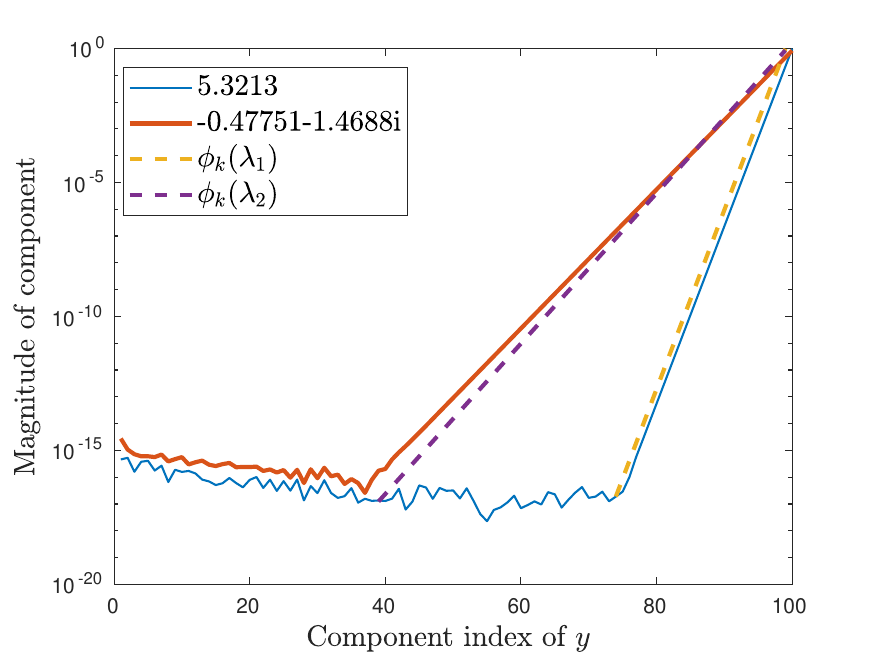}
  \caption{Example~\ref{ex:toepl1}. \emph{Left:} field of values and eigenvalues of $\mathscr{H}_d$, and eigenvalues
of $\mathscr{H}_d+\widehat{r}e_d ^T$. The arrows point to the eigenvalues used in the right plot.
\emph{Right:} absolute values of the eigenvector components for the selected
eigenvalues, together with the normalized vector of polynomial values $\phi_k$.
\label{fig:toeplitz1}}
\end{figure}

It is worth mentioning that whenever $\beta_i$ is small, 
the term $\beta_i f[N,\lambda_i]$ remains small. 
For instance, for the exponential,
$$
\beta_i \exp[N,\lambda_i] \approx 
\frac{|N_k|^{d-1}}{|N_{k,k+1}-\lambda_i|^{d-1}}
(N-\lambda_i I)^{-1}(\exp(N)-\exp(\lambda_i) I)
$$
for some $k<d$.
Let us assume, as in the previous example, that 
$\gamma_{\lambda_i,N}\approx |\lambda_i|>1$.
Then
$$
\frac{ |\exp(\lambda_i)|}{
(\gamma_{\lambda_i,N})^{d-1}}
\approx \frac{ |\exp(\lambda_i)|}{|\lambda_i|^{d-1}},
$$
which goes to zero as $d$ goes to infinity.

For our purposes, the analysis above shows that the likely erratic
spectral behavior in the sketched approach will not negatively affect the final approximation. Loosely speaking, the matrix structure takes care of purging the undesired eigenvalues. In this way, the ``bad'' portion of the generated space is implicitly deflated.

Motivated by the analysis above, we return to the relation $\alpha_i \beta_i = \frac{\prodM_d}{\omega'(\lambda_i)}$.
 Let $\cal J$ be the set of indexes $i$ such that $\beta_i$ is not tiny. Hence,
the sum over all eigenvalues can be limited to those in $\cal J$, that is
$$
\sum_{i=1}^d \alpha_i \beta_i f[M,\lambda_i] v \approx 
\sum_{i\in {\cal J}} \alpha_i \beta_i f[M,\lambda_i] v.
$$
A similar reduction can be written for (\ref{eqn:ferrnew2}), in terms of
the $d$th order divided differences $f[M, \lambda_1, \ldots, \lambda_d]$, namely
$f[M, \lambda_{i_1}, \ldots, \lambda_{i_{|{\cal J}|}}]$.
Hence, the set ${\cal J}$ identifies the ``effective'' eigenpairs on which the approximation error actually lives.

In the context of the sketched Arnoldi method, this shows that bounds like~\eqref{eq:bound_fsk_D} would not be expected to be descriptive of the actual behavior of the method if they were to rely on approximation properties on the whole field of values $F(\mathscr{H}_d+\widehat{r}e_d^T)$. Rather, the behavior of the method is dictated by a region that may be only slightly larger than the spectral region of $\mathscr{H}_d$. Precisely, in light of~\eqref{eq:bound_gMw} and the analysis above, to obtain a convergence estimate for the sketched Arnoldi method, we can take $D$ as any convex set containing $F(\mathscr{H}_d) \cup \{\lambda_i : i \in \cal J \}$. We therefore introduce a more descriptive region of the complex plane: Let $\delta \ge 0$ be such that all eigenvalues $\lambda_i, i \in {\cal J}$ satisfy
${\rm dist}(\lambda, F(\mathscr{H}_d)) < \delta$, and let
$$
F(\mathscr{H}_d,\delta):=\{ z\in {\mathbb C} : {\rm dist}(z, F(\mathscr{H}_d)) \leq \delta \};
$$ 
we will call this the {\it effective} field of values of $\mathscr{H}_d$.
Clearly, $F(\mathscr{H}_d)\subseteq F(\mathscr{H}_d,\delta)$.

For $\lambda_i$ with $i\not\in {\cal J}$, we let 
$\zeta_i= {\rm arg} \min_{z\in \partial F(\mathscr{H}_d)} |z- \lambda_i|$.
Then we define the following set of values
$$
\widehat \lambda_i = 
\left \{\begin{array}{cc}\lambda_i & i \in {\cal J}, \\
                     \zeta_i & i \not\in {\cal J}. 
        \end{array}\right .
$$
Then, according to the previous discussion,
$$
 \sum_{i=1}^d \frac{1}{\omega^\prime(\lambda_i)} f[\mathscr{H}_d,\lambda_i] v
\approx
 \sum_{i=1}^d \frac{1}{\omega^\prime(\widehat \lambda_i)} f[\mathscr{H}_d,\widehat \lambda_i] v.
$$
The definition of the $\widehat\lambda_i$'s thus allows us to move (rather than remove)
the inactive eigenvalues, while maintaining the number of terms in the divided difference. 
With these definitions, and using
$h(z):= f[z,\widehat \lambda_1,\dots,\widehat \lambda_d]$, (\ref{eqn:bound_h}), applied for $M=\mathscr{H}_d$, becomes
\begin{equation*}
\|h(\mathscr{H}_d)\| \leq (1+\sqrt{2}) \max_{z \in F(\mathscr{H}_d)} |h(z)| \le (1+\sqrt{2})\frac{\max_{\zeta \in F(\mathscr{H}_d,\delta)} |f^{(d)}(\zeta)|}{d!}.
\end{equation*}
The bound employs the region of the complex plane where the sketched
method is active, and this region can be
much smaller than the field of values of $\mathscr{H}_d+ \widehat{r}e_d^T$. By means of these new tools, we can obtain an approximate bound of the form
\begin{equation}\label{eq:bound_family}
\|f(\mathscr{H}_d + \widehat{r} e_d^T)e_1 - f(\mathscr{H}_d)e_1\| \lesssim \prodM_d 
(1+\sqrt{2})\frac{\max_{\zeta \in F(\mathscr{H}_d, \delta)} |f^{(d)}(\zeta)|}{d!}\, \|\widehat{r}\|,
\end{equation}
which is a family of analogues of~\eqref{eq:bound_gMw} depending on the effective field of values, or more precisely on $\delta$.

To summarize, our analysis shows that convergence of the sketched method will always depend on a region that is not much larger than that of $\mathscr{H}_d$ (and thus that of $A$): If sketched Ritz values far outside the spectral region of $\mathscr{H}_d$ occur, they simply do not play a role. Only those that are slightly outside need to be accounted for. Thus, one can typically expect at most a slight delay of convergence due to sketching. The role of $\gamma_{\lambda,N}$ is crucial, and its relation with the field of values of $N$ deserves further work, which will be postponed to a later investigation: while $\gamma_{\lambda,N} > 1$ whenever $\lambda$ is far away from $F(\mathscr{H}_d)$, we also observed that this condition might sometimes hold for eigenvalues close to or even within the field of values, indicating that there might be even more sketched Ritz values that do not significantly influence the error.

A by-product of this analysis is the following: If none of the sketched Ritz values lie outside of $F(\mathscr{H}_d)$, then we can take $\delta = 0$ above, and the convergence behavior of the sketched Arnoldi method simply depends on properties of $f$ on $F(\mathscr{H}_d,0) = F(\mathscr{H}_d)$, irrespective of how large $F(\widehat{H}_d + \widehat{t}e_d^T)$ is. This is in contrast to, e.g., the results of~\cite{CortinovisKressnerNakatsukasa2022} which rely on the Crouzeix-Palencia theorem~\cite{Crouzeix.Palencia.17} and therefore always depend on the whole field of values of $\widehat{H}_d + \widehat{t}e_d^T$, even if the sketched Ritz values are well-behaved.

\section{Robust computation of sketched approximations}\label{Robust computation of sketched approximations}
In this section we want to brief\/ly discuss two further issues related to the experimentally observed numerical (in)stability of sketched Krylov methods. On the one hand, we provide arguments that explain why the WS-Arnoldi version~\eqref{eq:sfom_whitened_hessenberg} of the sFOM approximation typically appears to be very robust in floating point arithmetic despite the inversion of the possibly very ill-conditioned matrix $T_d$ that is involved. On the other hand, we explain which potential problems can occur with the non-whitened variant~\eqref{eq:sfom_Hrh} based on a rank-one modification of $H_d$.

{
By starting from expression~\eqref{eq:sfom_whitened_hessenberg} for the sketched-whitened approximation, that is
$f_d^{\textsc{sk}} = U_d T_d^{-1} \widehat y$ with
$\widehat y= f(\widehat H_d + \widehat t e_d^T) e_1 \|Sb\|$,
we want to analyze the stability properties of $ T_d^{-1}\widehat y$ and show 
that { the solution accuracy} is not significantly influenced
by the possible ill-conditioning of $T_d$. { To this end, we {recall the} 
well-known fact that in general, triangular systems are often solved to much higher accuracy in floating point arithmetic than their condition number would suggest (although it is not possible to prove a precise, general results in this direction as---rather academic---counter examples exist). This observation goes back to Wilkinson~\cite{wilkinson1961error} and has since then been investigated many times; see, e.g., Chapter~8 in~\cite{higham2002accuracy} for an overview. In the following, we provide a result that is specifically tailored to our situation.}

 We first recall the following standard forward error bound for triangular systems; see, e.g., Equation~(2.6) in~\cite{HighamTriang}. 
Let $z=T_d^{-1} \widehat y$ and $\widetilde z$ be the exact and computed solutions, respectively.
Then
\begin{equation}\label{eqn:forward_error}
\frac{\|z-\widetilde z\|_\infty}{\|z\|_\infty} \le
(d+1)\, {\tt u}\,   {\rm cond}(T_d,z) + O(u^2),
\end{equation}
where 
$$
 {\rm cond}(T_d,z) = 
\frac{\|\, |T_d^{-1}|\, |T_d|\, |z|\, \|_\infty}{\|z\|_\infty}, 
$$
 and $|\,\cdot\,|$ is the component-wise absolute value.
Here {\tt u} is the unit roundoff.
\vskip 0.1in

\begin{proposition}~\label{Prop:triang}
Let ${\rm cond}_\infty (T)=\|T\|_\infty\|T^{-1}\|_\infty$ denote the infinity norm condition number.
Assume $T_d$ can be partitioned in such a way that 
$$
T_d=\begin{bmatrix} T_1 & T_2\\ 0 & \tau T_3\end{bmatrix},\quad
{\rm cond}_\infty (T_1)=c_1, \quad and \quad 
{\rm cond}_\infty (T_3)=c_3,
$$
and $\tau>0$. Both $T_1, T_3$ are square and nonsingular.
Then (\ref{eqn:forward_error}) holds
with 
$$
{\rm cond}(T_d,z) \le
\frac{\max\left \{ c_3\|z_2\|_{\infty}, \,
c_1 \|z_1\|_\infty + \||T_1^{-1}T_2|\,\|_\infty( 1+c_3)\|z_2\|_\infty 
\right \}}{\|z\|_\infty},
$$
where $z=[z_1;z_2]$ is partitioned accordingly with $T_d$.
\end{proposition}

\begin{proof}
Let $\widetilde T_3 = \tau T_3$. We spell out the definition of 
 {\rm cond}$(T_d,z)$ in terms of the partitioning of $T_d$.
We first write 
$T_d^{-1}=[T_1^{-1}, - T_1^{-1}T_2 \widetilde T_3^{-1}; 0, \widetilde T_3^{-1}]$.
We then have
\begin{eqnarray*}
\|\, |T_d^{-1}|\, |T_d|\,|z|\|_\infty &=&
\|\, |T_d^{-1}|\, 
\begin{bmatrix}
|T_1|\, |z_1| + |T_2|\, |z_2|\\
|\widetilde T_3|\,|z_2| 
\end{bmatrix}\,\|_\infty \\
&=&
\|\,   \begin{bmatrix}
|T_1^{-1}|\, (|T_1|\, |z_1| + |T_2|\, |z_2|) 
 + |T_1^{-1}T_2 \widetilde T_3^{-1}|\, |\widetilde T_3|\,|z_2| \\ 
|\widetilde T_3^{-1}||\widetilde T_3|\,|z_2| \\
\end{bmatrix}\|_\infty \\
&\le &
\|\begin{bmatrix}
c_1 \|z_1\|_\infty + \||T_1^{-1}T_2|\,\|_\infty( 1+c_3)\|z_2\|_\infty 
  \\ 
c_3\|z_2\|_{\infty} \\
\end{bmatrix} \|_\infty
\\
&= &
\max\{ c_3\|z_2\|_{\infty}, 
c_1 \|z_1\|_\infty + \||T_1^{-1}|\,  |T_2|\,\|_\infty( 1+c_3)\|z_2\|_\infty \}
\end{eqnarray*}
\end{proof}

Proposition~\ref{Prop:triang} shows that, as long as the ill-conditioning of $T_d$ is encoded in its right-bottom corner {\color{black} as the multiplicative term $\tau$}, the entry-wise condition number {\rm cond}$(T_d,z)$, and thus the forward error in~\eqref{eqn:forward_error}, is bounded solely by the condition numbers of well-behaved submatrices and exact arithmetic quantities. Moreover, notice that $T_3$ in Proposition~\ref{Prop:triang} does not need to be a matrix but it can also be just a scalar. In this case it holds $c_3=1$ {\color{black} and $\tau=\tau_d$}.

The assumptions in Proposition~\ref{Prop:triang} are easily met in our sketched-and-truncated framework. Indeed, the non-orthogonal basis $U_d$ \emph{gradually} becomes more ill-conditioned as the iteration progresses and (almost) linear dependence occurs mostly between ``later'' basis vectors. To understand this, consider the extreme case of no orthogonalization at all: Then, the sequence of Krylov basis vectors converges towards the dominant eigenvector of $A$, so that later basis vectors will all point in almost the same direction, while staying well linearly independent with respect to the first basis vectors. A similar phenomenon is preserved when employing a partial orthogonalization with small truncation parameter. 

As the sketching matrix $S$ distorts inner products in a controlled manner due to~\eqref{eq:sketch_innerproduct}, this property of $U_d$ is inherited by the sketched basis $SU_d$ for which we perform the QR decomposition. This observation implies that the severe ill-conditioning of $T_d$ is mostly encoded in a bottom right subblock, the same exact scenario depicted in Proposition~\ref{Prop:triang}. 
}

\begin{remark}\label{rem:balancing}
Considering~\eqref{eq:sfom_Hrh}, note that the matrix $H_d+re_d^T$ is often not well-balanced, in the sense that the norm of its last column is typically much larger than that of all other columns. It is known since the 1970s that balancing can have an influence on the approximation accuracy for matrix functions; e.g., when using techniques such as Pad\'e approximation~\cite{fair1970pade,ward1977numerical}, although the effect is difficult to rigorously quantify; cf.,~e.g., sections~3 and~6 in~\cite{almohy2011computing}. 

We turn to the specific case of scaling-and-squaring methods for the matrix exponential. A first step in such methods is the computation of a parameter $\gamma$ such that $\|2^{-\gamma}(H_d+re_d^T)\| < 1$. When $H_d+re_d^T$ is poorly balanced, the value of $\gamma$ will be almost solely determined by its last column, and the entries of all other columns in $2^{-\gamma}(H_d+re_d^T)$ might become tiny. This and related effects are known under the term \emph{overscaling} in the matrix function literature; see, e.g., section~1 in~\cite{almohyhigham2010} or\cite{dieci2000pade}. Clearly, this can result in a noticeable decrease in the accuracy of the method in the presence of round-off error. In contrast, the matrix $\widehat{H}_d + \widehat{t}e_d^T$ from~\eqref{eq:sfom_Hrh} will typically be better balanced due to the similarity transformation with $T_d$ and thus less prone to overscaling effects.
\end{remark}

{\color{black} We next report the results of an illustrative example. For the sake of clarity we report an algorithmic description of the whitened-sketched Arnoldi method that we use as Algorithm~\ref{alg:whitening_krylov} in Appendix~\ref{appendix:algorithm}. We stress that in our implementation the QR decomposition employed in the whitening procedure was updated from one iteration to the next instead of recomputing it, therefore causing only a modest computational cost increase, for $s$ significantly smaller than $n$.}

\begin{example}\label{ex:whitening}%
{\color{black}In this example,} $A$ is obtained by discretizing the convection-diffusion operator
\begin{equation}\label{eq:condiff_op}
    \mathcal{L}(u)=-\nu\Delta u+\vec{w}\cdot u,
\end{equation}
on the unit square $[0,1]^2$ by centered finite differences. We use viscosity parameter $\nu = 10^{-2}$, the convection field $\vec{w}=(\frac{3}{2}y(1-x^2),-3x(1-y^2))$, and $N = 50$ discretization points in each spatial direction. The resulting matrix $A$ is highly non-normal with extremely ill-conditioned eigenvector basis, $\kappa(X) \approx 1.5 \cdot 10^{22}$. We aim to approximate the action of the matrix exponential $\exp(-A)b$, where $b$ is the normalized vector of all ones. 

We begin by comparing different possible implementations of sFOM, using a truncation parameter $k = 2$ in the generation of the non-orthogonal Krylov basis $U_d$ and a sketching dimension $s = 400$. Specifically, we compare the whitened-sketched sFOM approximation~\eqref{eq:sfom_whitened_hessenberg} with the version~\eqref{eq:sfom_Hrh} based on the rank one modification $H_d +re_d^T$. Due to the bad balancing of $H_d + re_d^T$, the computation of $\exp(H_d+re_d^T)$ might be extremely sensitive to round-off, also depending on which algorithm is used for the matrix exponential; cf.~also the discussion in Remark~\ref{rem:balancing}. We therefore compare two different versions: one which directly uses the built-in MATLAB function \texttt{expm} (which implements the scaling-and-squaring method from~\cite{almohyhigham2010}) and one which first computes a Schur decomposition of $A$ and then applies \texttt{expm} only to the triangular factor. Note that this can lead to better numerically stability for two reasons: First, the triangular factor of the Schur decomposition will in general be better balanced than $H_d+re_d^T$, and second, \texttt{expm} contains several implementation tricks to specifically improve the behavior for triangular inputs; see section~2 in~\cite{almohyhigham2010}. For the whitened-sketched Arnoldi method, due to the better balancing, using either \texttt{expm} directly or working with a Schur decomposition gives virtually identical results so that we do not report separate curves here.

\begin{figure}[tb]
    \centering
    \includegraphics[width=.49\textwidth]{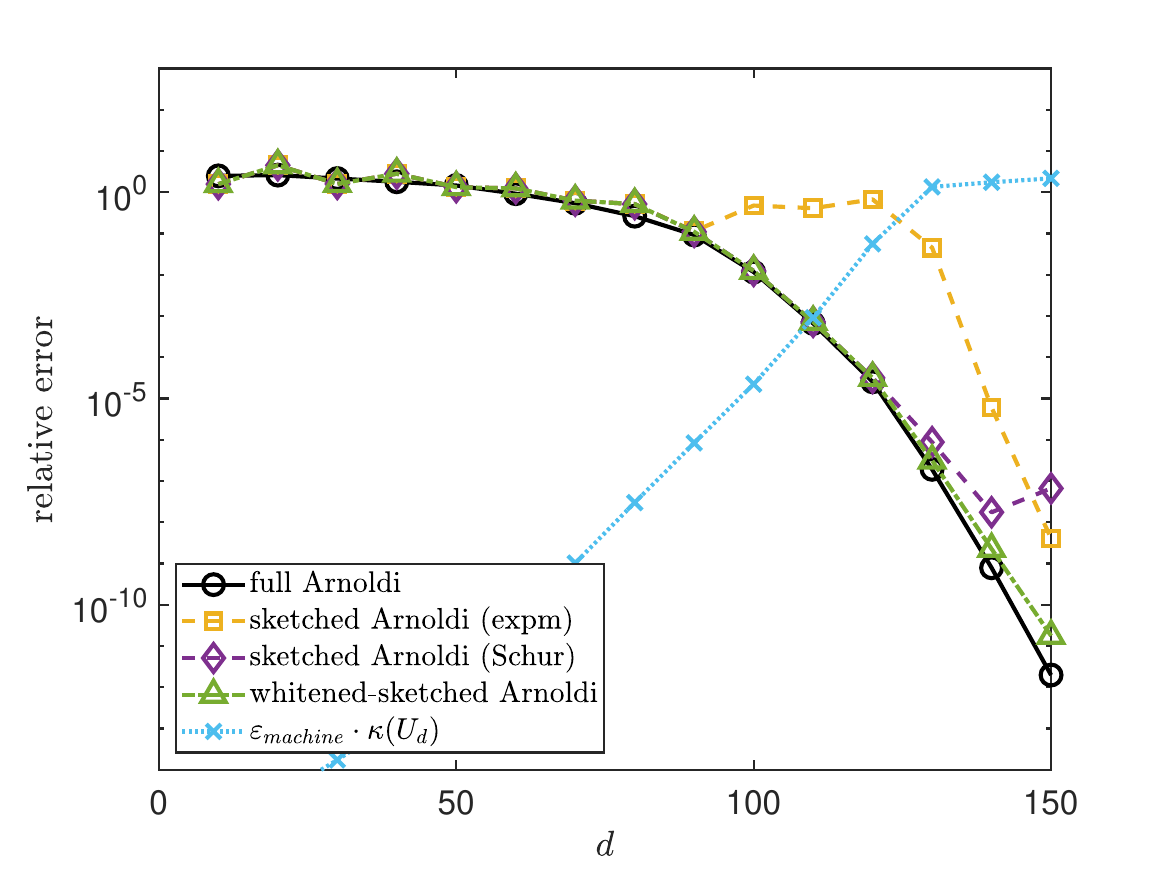}
    \includegraphics[width=.49\textwidth]{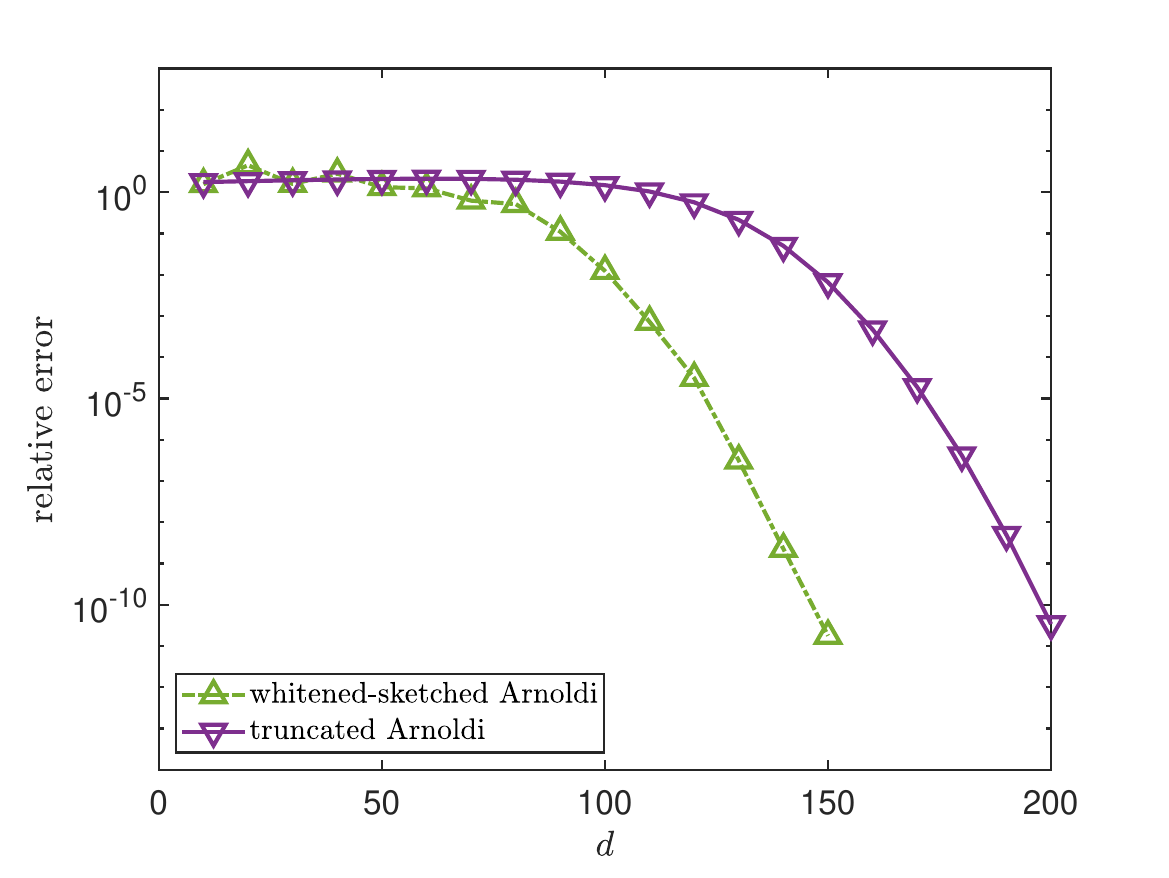}
     \caption{\emph{Left:} Comparison of different versions of the sFOM approximation with full Arnoldi for approximating the action of the matrix exponential. The conditioning of the non-orthogonal Krylov basis $U_d$ is also depicted. \emph{Right:} Comparison of ``best'' sFOM implementation (i.e., whitened-sketched) with truncated Arnoldi. In both cases, the truncation parameter is $k = 2$ and the sketching parameter is $s = 400$.}
     \label{fig:matfun_comparison}
\end{figure}

The results of this experiment are depicted on the left-hand side of Figure~\ref{fig:matfun_comparison}. It can be clearly observed that the whitened-sketched method closely follows the convergence of the full Arnoldi method, with only a very small deviation towards the end of the process. In contrast, both versions of sFOM based on the rank-one modification~\eqref{eq:sfom_Hrh} show signs of instability. While all sketched methods behave identical for the first 90 iterations, the version based on \texttt{expm} becomes very unstable after that and shows increasing/stagnating error norms for about 50 iterations before continuing to converge after that. The version based on the Schur decomposition longer behaves similarly to the whitened-sketched method, but also starts to deviate around the 130th iteration, when reaching a relative error norm of about $10^{-8}$.

It is a common conception (see, e.g.,~\cite{NakatsukasaTropp2021}) that convergence of sketched methods typically deteriorates once the non-orthogonal basis becomes numerically linearly dependent (i.e., once it has a condition number larger than $\varepsilon_{machine}^{-1}$, the reciprocal of the unit round-off). Our experiment clearly reveals that this is not necessarily true and that there are also other factors which have an influence. For sFOM using \texttt{expm}, stability problems occur long before the basis becomes severely ill-conditioned, and the whitened-sketched method continues to converge smoothly also once this happens. Only for sFOM using a Schur decomposition, the linear dependence of the basis and occurrence of stability problems coincide.

In the right part of Figure~\ref{fig:matfun_comparison}, we compare the best performing of the sFOM variants, i.e., the whitened-sketched version~\eqref{eq:sfom_whitened_hessenberg}, to truncated Arnoldi~\eqref{eq:trfom} using the same truncation parameter $k=2$. We can observe that in truncated Arnoldi, convergence is significantly delayed. In the initial 110 iterations of the method, a plateau of stagnating (or even slightly increasing) error norms occurs. After that, convergence begins, and the convergence rate is only slightly slower than that of the full (and whitened-sketched) method. The truncated method reaches the target relative error norm $10^{-11}$ after 200 iterations. Thus, the sketched-whitened method uses 25\% fewer matrix-vector products in this case.
\end{example}

\section{Conclusions}\label{Conclusions}
We have provided a new analysis of sketched Krylov subspace methods. In particular, we have derived a new \emph{sketched Arnoldi relation}, which applies to any Krylov method that combines subspace embeddings with a truncated orthogonalization process. This relation can be used to obtain new insights into the behavior of these methods and better explains some phenomena that are frequently observed, e.g., that methods continue to perform well even if they produce spurious Ritz values far outside the field of values of $A$.

To demonstrate this, we have specifically focused on the application of sketching to the approximation of $f(A)b$, the action of a matrix function on a vector. We have derived new formulas for expressing certain rank-one modifications of a matrix function via divided differences and then used those to compare sketched and truncated Krylov approximations to the standard (full) Arnoldi approximation. In particular, we could prove that for functions like the exponential, for growing Krylov dimension $d$, the sketched Arnoldi approximation is guaranteed to converge to the full Arnoldi approximation, and the nature of our bound suggests that one can expect this convergence to take place at roughly the same speed as convergence of the full Arnoldi approximation to $f(A)b$.

Our focus in this work was on deepening the theoretical understanding of existing sketching approaches, not on introducing new algorithmic concepts. Therefore, we did not perform extensive, large-scale numerical experiments, as thorough experimental studies illustrating the potential of sketched Krylov methods for matrix functions already exist in the literature; see, e.g., \cite{CortinovisKressnerNakatsukasa2022,GuettelSchweitzer2022}. 

In~\cite{PalittaSchweitzerSimoncini2023bis} we employ the sketched Arnoldi approximation derived here for analyzing and implementing sketched and truncated Krylov subspace methods for efficiently solving matrix Sylvester and Lyapunov equations.

\bmsection*{Acknowledgments}
The first and third authors are members of the INdAM Research Group GNCS that partially supported this work through the funded project GNCS2022 ``Tecniche Avanzate per Problemi Evolutivi: Discretizzazione, Algebra Lineare Numerica, Ottimizzazione''  (CUP\_E55F22000270001).

\bmsection*{Financial disclosure}

None reported.

\bmsection*{Conflict of interest}

The authors declare no potential conflict of interests.

\bibliography{refs}

\appendix

\bmsection{Pseudocode for sketched-and-truncated Arnoldi method}\label{appendix:algorithm}
We report here the pseudocode illustrating the sketch-and-truncated Krylov subspace method for the approximation of $f(A)b$. 

We stress that adopting the whitening (lines \ref{white:line1}--\ref{white:line2} of Algorithm~\ref{alg:whitening_krylov}) does not remarkably increase the overall computational cost of the procedure. Indeed, assuming the sketched basis $SU_d$ and its QR factors $Q_d$ and $T_d$ have been stored (these are small dimensional matrices), line~\ref{white:line1} requires the application of the sketching to $u_{d+1}$, whose cost depends on the selected sketching but it is often only polylogarithmic in $n$, and the update of the QR factorization of $[SU_d,Su_{d+1}]$, which can be computed in $O(sd)$ flops by, e.g., a Gram-Schmidt procedure.
Moreover, we can fully take advantage of the upper Hessenberg, banded structure of $H_d$ and the triangular pattern of $T_d$ to reduce the cost of updating $\widehat H_d$ in line~\ref{white:line2}; {\color{black} see section 5 in~\cite{PalittaSchweitzerSimoncini2023bis}}.

Notice that Algorithm~\ref{alg:whitening_krylov} can be additionally enhanced with a two-pass strategy to avoid storing the whole basis $U_d$. If this strategy is implemented, step 10 is postponed to after the end of the loop, and step 11 is reduced to store only the last $k+1$ vectors. Finally, an ad-hoc stopping criterion tailored to the function $f$ at hand can be included in the for-loop.

\begin{algorithm}[t]
\caption{Sketched-and-truncated Krylov method for $f(A)b$}\label{alg:whitening_krylov}
\begin{algorithmic}[1]
\smallskip

\Statex \textbf{Input:} $A\in\mathbb{R}^{n\times n}$, $b\in\mathbb{R}^{n}$, $S\in\mathbb{R}^{s\times n}$, integers $0<k<\text{maxit}\ll n$\  
\Statex \textbf{Output:} $f_d^{\textsc{sk}}\approx f(A)b$
\State Set $U_1=u_1=b/\|b\|$, $Q_1=SU_1/\|SU_1\|$, $T_1=\|SU_1\|$
 \For{$d=1,\ldots,\text{maxit}$}
\State Compute $\widetilde u=Au_d$
\For{$i=\max\{1,d-k+1\},\ldots,d$}
\State Set $\widetilde u=\widetilde u- u_ih_{i,d}$, where $h_{i,d}= \widetilde u^Tu_i$ 
\EndFor
\State Set $h_{d+1,d}=\|\widetilde u\|$ and $u_{d+1}=\widetilde u/h_{d+1,d}$
\State Update skinny QR: $Q_{d+1}T_{d+1}=[SU_{d},Su_{d+1}]$ \label{white:line1}
\State Update $\widehat H_d=T_dH_dT_d^{-1}$ and set $\widehat t=(h_{d+1,d}/\tau_d)t$ \label{white:line2}
\State Compute $f_d^{\textsc{sk}}=U_d(T_d^{-1}(f(\widehat{H}_d + \widehat{t}e_d^T)e_1\|Sb\|))$
\State Set $U_{d+1}=[U_d,u_{d+1}]$
\EndFor

\end{algorithmic}
\end{algorithm}

\bmsection{Proofs.}\label{appendixB}
{\color{black} In this section we report the proofs of Theorem~\ref{th:errorf} and Theorem~\ref{th:ferror_DD}.

\begin{proof}[Proof of Theorem~\ref{th:errorf}]
We follow the steps in Theorem~2.1 of \cite{IlicTurnerSimpson2010}. We first write
\begin{eqnarray*}
& &\hspace{-2cm} (zI - M - v e_d^T)^{-1} e_1 - (zI - M)^{-1} e_1 \\
&=& (zI - M)^{-1} [ (zI - M)(zI - M - v e_d^T)^{-1}e_1 - e_1 ]\\
&=& (zI - M)^{-1} v\,\, [e_d^T (zI - M - v e_d^T)^{-1} e_1].
\end{eqnarray*}
Using the eigendecomposition of $M + v e_d^T$, we have 
$$e_d^T (zI - M - v e_d^T)^{-1} e_1=\sum_{i=1}^d \frac{\alpha_i \beta_i}{z-\lambda_i}.$$
Let $\Gamma$ be a closed curve enclosing the eigenvalues of $M + v e_d^T$ and
of $M$. Then
\begin{align}
&\hspace{-1.5cm} f(M + v e_d^T)e_1 - f(M)e_1 \nonumber\\
&=
\frac 1 {2\pi \imath} \int_\Gamma f(z) \left ( (zI - M - v e_d^T)^{-1} e_1 - (zI - M)^{-1} e_1 \right ) dz\nonumber\\
&= \frac 1 {2\pi \imath} 
\int_\Gamma f(z) (zI - M)^{-1}  v (e_d^T (zI - M - v e_d^T)^{-1} e_1) dz\label{eq:integral_representation}\\
&= \frac 1 {2\pi \imath} \sum_{i=1}^d \alpha_i \beta_i \int_\Gamma f(z)\frac 1 {z-\lambda_i} (zI - M)^{-1}dz\ v\nonumber\\
&= \sum_{i=1}^d \alpha_i \beta_i f[M, \lambda_i] v\nonumber\\
&= g_v(M) v,\label{eq:closed_form_g}
\end{align}
where we used the contour integral representation of divided differences (see, e.g., Equation~(51) in~\cite{DeBoor2005}) in the second to last equality.
\end{proof}

\begin{remark}
We note that in Lemma~2.2 of~\cite{beckermann2018low} a conceptually different representation for rank-one updates of matrix functions was found---also based on the integral representation~\eqref{eq:integral_representation}---but that this representation is less useful for our analysis.
\end{remark}

We then proceed with showing Theorem~\ref{th:ferror_DD}.

\begin{proof}[Proof of Theorem~\ref{th:ferror_DD}]
Following the proof of Theorem~\ref{th:errorf}, formula (\ref{eq:integral_representation}) dwells with
the quantity
\begin{equation}\label{eq:integral_proof42}
	{\cal E }:=\int_\Gamma f(z)(zI-M)^{-1} v (e_d^T(zI - M - v e_d^T)^{-1}e_1) \ \, dz.
\end{equation}
Now, let $\widehat M = M + v e_d^T$ and $\widehat M_z = z I - \widehat M$, where $z \in {\mathbb C}$ is such that $z I - \widehat M$ is nonsingular. Further denote by $\phi_d(z) = z^d + c_{d-1}z^{d-1}+\dots ü c_1z+c_0$ the characteristic polynomial of $\widehat{M}_z$, so that by the Cayley-Hamilton theorem, we have $\phi_d(\widehat M_z) = \widehat M_z^d + c_{d-1} \widehat M_z^{d-1} + \ldots + 
c_1 \widehat M_z + c_0 I = 0$, with $c_0=(-1)^d {\rm det}(\widehat M_z) $. Then we have
$$
  \widehat M_z^{-1} =
\frac {-1}{c_0}\left (\widehat M_z^{d-1} + c_{d-1} \widehat M_z^{d-2} + \ldots +  c_1 I_d \right ).
$$
Since $M$, and thus $\widehat M_z$, is upper Hessenberg, it holds
$$
e_d^T \widehat M_z^{-1}e_1 = \frac{(-1)^{d+1}}{ {\rm det}(\widehat M_z)} e_d^T \widehat M_z^{d-1} e_1= \frac{(-1)^{d+1}}{ {\rm det}(\widehat M_z)} \prod_{j=1}^{d-1}\widehat M_{j+1,j}=\frac{(-1)^{2d}}{ {\rm det}(\widehat M_z)} \prod_{j=1}^{d-1} M_{j+1,j}=\frac{1}{ {\rm det}(\widehat M_z)} \prod_{j=1}^{d-1}M_{j+1,j}.
$$
Moreover,
$$
{\rm det}(\widehat M_z) = \prod_{j=1}^d \lambda_j(\widehat M_z) = \prod_{j=1}^d (z-\lambda_j),
$$
where $\lambda_j$ are the eigenvalues of $\widehat M$.
Therefore,
$$
e_d^T(zI - M - v e_d^T)^{-1}e_1 = 
\prod_{j=1}^{d-1}M_{j+1,j}
\prod_{i=1}^d (z - \lambda_i)^{-1}.
$$
We can thus write the integral in~\eqref{eq:integral_proof42} as
$$
{\cal E }=\prod_{j=1}^{d-1}M_{j+1,j}
\int_\Gamma \frac{f(z)}{\omega(z)}(zI-M)^{-1} v \,dz.
$$
To obtain (\ref{eqn:ferrnew1}), we write the partial
fraction expansion of $1/\omega(z)$, that {for distinct eigenvalues is given by}
$$
\frac{1}{\omega(z)}=\sum_{i=1}^d 
\frac{1}{\omega^\prime(\lambda_i)} \frac{1}{z-\lambda_i},
$$
where the prime denotes differentiation of $\omega(z)$ with respect to $z$,
yielding
$$
\int_\Gamma \frac{f(z)}{\omega(z)}(zI-M)^{-1} v  \,dz =
\sum_{i=1}^d  \frac{1}{\omega^\prime(\lambda_i)} 
\int_\Gamma \frac{f(z)}{z-\lambda_i} (zI-M)^{-1} v \, dz.
$$
Using once again the contour integral representation of divided 
differences and collecting the other terms, the formula (\ref{eqn:ferrnew1}) is proved.

To obtain (\ref{eqn:ferrnew2}) we use the integral representation~\eqref{eq:integral_higher_order_divdiff} to write
$$
\int_\Gamma \frac{f(z)}{\omega(z)}(zI-M)^{-1}  dz =
f[M,\lambda_1, \ldots, \lambda_d],
$$
or, equivalently, we observe that 
$$
 \frac 1 {2\pi \imath} 
 \sum_{i=1}^d \frac{1}{\omega^\prime(\lambda_i)} f[z,\lambda_i]  = 
 f[z,\lambda_1, \ldots, \lambda_d],
$$
see, e.g., Equation~(1.1) in~\cite{Curtiss.62}.
\end{proof}

}
\end{document}